\documentclass[11pt]{article}

\usepackage[margin=1in]{geometry}
\usepackage{amsmath}
\usepackage{amsthm}
\usepackage{comment}
\usepackage{amssymb}
\usepackage{dsfont}
\usepackage{enumerate}
\usepackage{tikz}
\usepackage{hyperref}
\usepackage{sidecap}
\usepackage{caption,subcaption}
\usepackage[capitalise]{cleveref} 
\usepackage{setspace}
\usetikzlibrary{shapes,snakes}
\usetikzlibrary{shapes.geometric} 

\newtheorem{theorem}{Theorem}
\newtheorem{proposition}{Proposition}
\newtheorem{lemma}[proposition]{Lemma}
\newtheorem{corollary}{Corollary}
\theoremstyle{definition}
\newtheorem{definition}[proposition]{Definition}
\numberwithin{proposition}{section} 

\renewcommand{\subset}{\subseteq}
\renewcommand{\hat}{\widehat}
\renewcommand{\epsilon}{\varepsilon}

\def\dist{{\rm dist}}
\def\diam{{\rm diam}}

\def\wherespace{\quad\text{where}\quad}
\def\foreachspace{\quad\text{for each}\quad}
\def\andspace{\quad\text{and}\quad}

\def\<{\langle}
\def\>{\rangle}
\def\({\left(}
\def\){\right)}

\def\calB{\mathcal{B}}
\def\calC{\mathcal{C}}

\def\calG{\mathcal{G}}

\def\calI{\mathcal{I}}
\def\calJ{\mathcal{J}}

\def\calP{\mathcal{P}}
\def\calS{\mathcal{S}}
\def\calT{\mathcal{T}}
\def\calU{\mathcal{U}}
\def\calV{\mathcal{V}}
\def\calW{\mathcal{W}}
\def\calX{\mathcal{X}}

\def\C{\mathbb{C}}

\def\N{\mathbb{N}}

\def\R{\mathbb{R}}

\def\T{\mathbb{T}}
\def\Z{\mathbb{Z}}

\numberwithin{equation}{section}

\title{Multiscale estimates for the condition number of \\ non-harmonic Fourier matrices}
\author{Weilin Li\footnote{CUNY City College. Email: wli6@ccny.cuny.edu}}

\begin{document}
\maketitle

\begin{abstract}
	This paper studies the extreme singular values of non-harmonic Fourier matrices. Such a matrix of size $m\times s$ can be written as $\Phi=[ e^{-2\pi i j x_k}]_{j=0,1,\dots,m-1, k=1,2,\dots,s}$ for some set $\calX=\{x_k\}_{k=1}^s$. Its condition number controls the stability of inversion, which is of great importance to super-resolution and nonuniform Fourier transforms. Under the assumption $m\geq 6s$ and without any restrictions on $\calX$, the main theorems provide explicit lower bounds for the smallest singular value $\sigma_s(\Phi)$ in terms of distances between elements in $\calX$. More specifically, distances exceeding an appropriate scale $\tau$ have modest influence on $\sigma_s(\Phi)$, while the product of distances that are less than $\tau$ dominates the behavior of $\sigma_s(\Phi)$. These estimates reveal how the multiscale structure of $\calX$ affects the condition number of Fourier matrices. Theoretical and numerical comparisons indicate that the main theorems significantly improve upon classical bounds and recover the same rate for special cases but with relaxed assumptions. 
\end{abstract}

\medskip
\noindent
{\bf 2020 Math Subject Classification:} 15A12, 15A60, 42A05, 42A15, 65F22


\section{Introduction}

\subsection{Motivation}

\label{sec:background}

For any set $\calX=\{x_k\}_{k=1}^s\subset \T:=\R/\Z$ and natural number $m\geq s$, a {\it (non-harmonic) Fourier matrix} of size $m\times s$ is defined as 
$$
\Phi
:=\Phi(m,\calX)
:=\Big[ e^{-2\pi i j x_k} \Big]_{j=0,1,\dots,m-1, \, k=1,2,\dots,s}.
$$
This definition generalizes the discrete Fourier transform matrix, whereby $m=s$ and $\calX$ consist of $s$ equally spaced points in $\T$. Throughout the expository portions of this paper, we will implicitly assume that $|\calX|=s$, and we impose $m\geq s>1$ to avoid trivialities. 

Fourier matrices are classical objects that appear in numerous areas of mathematics. They provide a fundamental connection between linear algebra and trigonometric interpolation, which can be traced back to the work of Newton and Lagrange. They are matrix representations of the Fourier transform, so they naturally appear in the analysis of Fourier series \cite{zygmund1959trigonometric}, exponential sums \cite{young1981introduction}, and nonuniform Fourier transforms \cite{dutt1993fast}. Since $\Phi$ is also a Vandermonde matrix, it has full rank whenever $m\geq s$. When the rows of $\Phi$ are viewed as elements of $\C^s$, then the squared extreme singular values are the upper and lower frame constants \cite{casazza2012finite,duffin1952class}. 

For numerical applications, we require quantitative estimates for the extreme singular values of $\Phi$ to ensure that it can be inverted in numerical schemes without incurring significant error. We provide two motivational examples. The first is a simplified model for one-dimensional super-resolution, whereby $\calX$ represents the support of a discrete measure and $\Phi$ is the matrix representation of the Fourier transform of the measure sampled at $m$ consecutive integers. The problem is to recover $\calX$ through noisy Fourier data, and even though this is a nonlinear inverse problem, the condition number of $\Phi$ is closely related to performance of algorithms \cite{li2021stable,li2020super,li2022stability} and fundamental limits of recovery \cite{donoho1992superresolution,demanet2015recoverability,li2021stable,batenkov2021super}. The second is a connection to the nonuniform discrete Fourier transform of Type I. Here, $\Phi$ is the NUDFT Type I matrix that evaluates  the first $m$ Fourier series coefficients of the measure on $\calX$ for the periodic domain $\T$, see \cite{dutt1993fast,plonka2018numerical} for an overview. This is an over-determined linear system, so the condition number of $\Phi$ is pertinent to the least squares solution of fitting data by exponential sums. 

While classical papers such as \cite{gautschi1963inverses,cordova1990vandermonde,berman2007perfect} concentrated on square matrices, tall ones, where $m$ may be significantly larger than $s$, tend to be better conditioned. Tall Fourier matrices have received considerable interest recently \cite{liao2016music,moitra2015matrixpencil,liao2016music,aubel2019vandermonde,li2021stable,batenkov2020conditioning,batenkov2021single,kunis2020smallest,kunis2021condition}. This is partly due to modern applications in signal and image processing, where rectangular matrices appear more frequently, since $m$ represents the number of measurements or parameters, while $s$ corresponds to the number of constraints, see \cite{donoho1992superresolution,fannjiang2010compressed,benedetto2020super,li2020super,li2022stability,chui2022super} and references therein. It is worth noting that parallel to this line of research, the approximation properties of trigonometric interpolation in the $m\geq s$ regime has received interest \cite{li2021generalization,xie2022overparameterization,ren2022a} due to connections with over-parameterization in machine learning. 

The condition number $\kappa(\Phi(m,\calX))$ greatly depends on the ``Rayleigh length" $\frac 1 m$ versus the ``geometry" of $\calX$. The latter can be partially described by the {\it minimum separation} of $\calX$, defined as 
$$
\Delta(\calX):=\min_{j\not=k} |x_j-x_k|_\T, \wherespace 
|t|_\T := \min_{n\in\Z} |t-n|.
$$
Letting $\sigma_k(\Phi)$ denote the $k$-th largest singular value of $\Phi$, it was shown in \cite{aubel2019vandermonde} that
\begin{equation}
	\label{eq:wellseparated}
	\text{if}\quad \Delta(\calX)>\frac 1 m, \quad\text{then} \quad 
	\sqrt {m -\frac{1}{\Delta(\calX)}}
	\leq \sigma_s(\Phi(m,\calX))
	\leq \sigma_1(\Phi(m,\calX))
	\leq \sqrt {m +\frac{1}{\Delta(\calX)}}.
\end{equation}
The intuition behind this inequality is that the columns of $\Phi$ are almost orthogonal. This result and a closely related one in \cite{moitra2015matrixpencil}, are proved using analytic number theory methods. 

On the other hand, when $\Delta(\calX) \leq \frac{1}{m}$, simple numerical experiments, see \cite{li2021stable,batenkov2020conditioning}, show that $\sigma_s(\Phi(m,\calX))$ does not follow the behavior in \eqref{eq:wellseparated}. This makes intuitive sense since if $|x_k-x_\ell|_\T$ is small, then the $k$-th and $\ell$-th columns of $\Phi$ are highly correlated, which results in a large condition number. If it is significantly larger than $\sqrt s$, then the smallest singular value is the culprit because we have the trivial estimate $\sigma_1(\Phi)\leq \|\Phi\|_F=\sqrt {ms}$, where $\|\cdot\|_F$ denotes the Frobenius norm. 

Accurate bounds for the smallest singular value have been obtained under specific scenarios, namely when $\calX$ can be partitioned in subsets called ``clumps", where each clump is contained in an interval whose length is on the order of $\frac 1 m$. In contrast to the separation condition required in \eqref{eq:wellseparated}, without any conditions on $\Delta(\calX)$ relative to $m$, the results in \cite{li2021stable,batenkov2020conditioning,kunis2020smallest,batenkov2021single,batenkov2021spectral} roughly state that if each clump has cardinality at most $\lambda$ and the clumps are sufficiently far away from each other, then 
\begin{equation}
	\label{eq:clumps}
	c({\lambda, s})\sqrt m \, (m \Delta(\calX))^{\lambda-1}
	\leq \sigma_s(\Phi(m,\calX))
	\leq C({\lambda, s}) \sqrt m \, (m \Delta(\calX))^{\lambda-1}.
\end{equation}
Since the exponent $\lambda-1$ may be significantly smaller than $s-1$, this bound shows that $\sigma_s(\Phi)$ depends on the local geometry of $\calX$. It captures the intuition that columns of $\Phi$ which correspond to different clumps are almost orthogonal with respect to each other, so we expect the conditioning of $\Phi$ to only depend on each clump separately.

\subsection{A motivational multiscale example} 
\label{sec:motivation}

To better illustrate the limitations of prior work, let us consider a typical set with multiscale structure such as 
\begin{gather}
	\begin{split}
		\calX&:=\calX_1\cup \calX_2\cup \calX_3,\wherespace \\
		\calX_1 := \{0,\tfrac 1 {90}, \tfrac 2{90}, \tfrac 3{90}\}, \quad
		&\calX_2 = \tfrac 1 3 + \{0,\tfrac 1 {200}, \tfrac 2  {200}\}, \andspace
		\calX_3 = \tfrac 2 3 + \{0, \tfrac 1 {500}\}. 
	\end{split}
	\label{eq:Xmotivation}
\end{gather}
We have defined in $\calX$ this way to emphasize that its three disjoint subsets $\calX_1$, $\calX_2$, and $\calX_3$ each have significantly different minimum separations and should be treated as sets with completely different scales. The set $\calX$ is shown in Figure \ref{fig:motivational}. Our problem is to determine behavior of $\sigma_s(\Phi(m,\calX))$ as a function of $m$. Clearly $\Phi(m,\calX)$ only has full rank when $m\geq s=9$ and \eqref{eq:wellseparated} kicks in when $m>\frac 1 {\Delta(\calX)} =500$. What about the missing range $m\in [9,500]\cap\N$?

For the range $m\in [9,600]\cap\N$, none of the bounds in \cite{li2021stable,batenkov2020conditioning,kunis2020smallest,batenkov2021single} are applicable. The reason is that while $\calX$ consists of three ``clumps" $\calX_1$, $\calX_2$, and $\calX_3$, they are too close to each other and do not satisfy these theorems' assumptions, or such theorems have implicit constants in their separation criterion that cannot be explicitly determined. It is important to remark that the aforementioned papers concentrated on the super-resolution limit, whereby either $m$ is sufficiently large and there is a sequence of $\calX_m$ for which $\Delta(\calX_m)\to 0$, or alternatively, $m\to\infty$ and $\Delta(\calX_m)\to 0$ with some relationship between $m$ and $\Delta(\calX_m)$. Hence, it is not that surprising they cannot be directly used for fixed $\calX$ and $m$. 

\begin{figure}[h]
	\centering 
	\begin{subfigure}{0.49\textwidth}
		\centering
		\includegraphics[height=0.22\textheight]{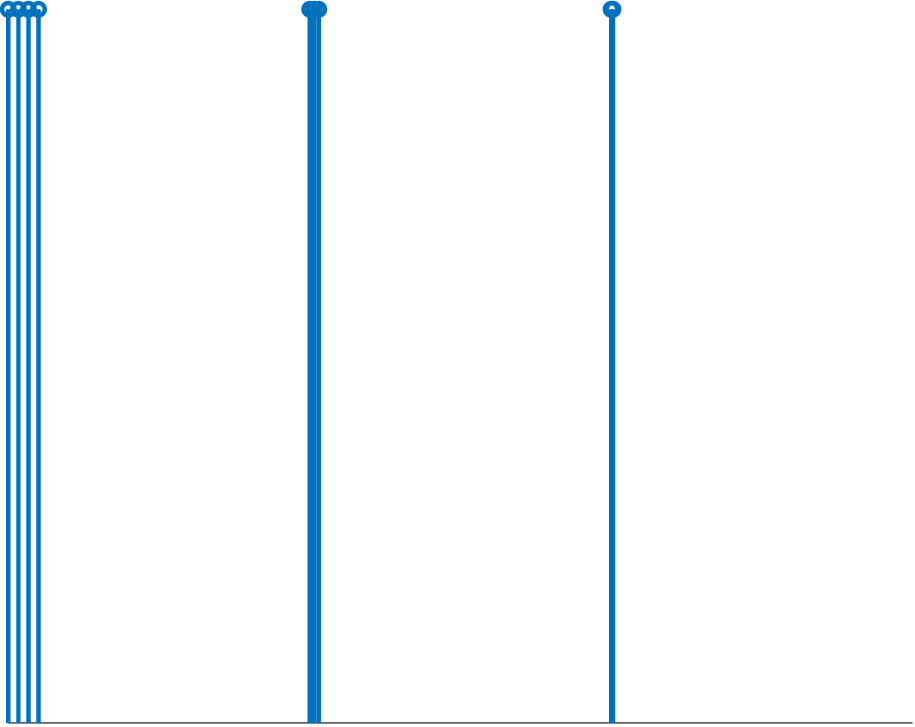}
	\end{subfigure}
	\begin{subfigure}{0.49\textwidth}
		\centering 
		\includegraphics[height=0.25\textheight]{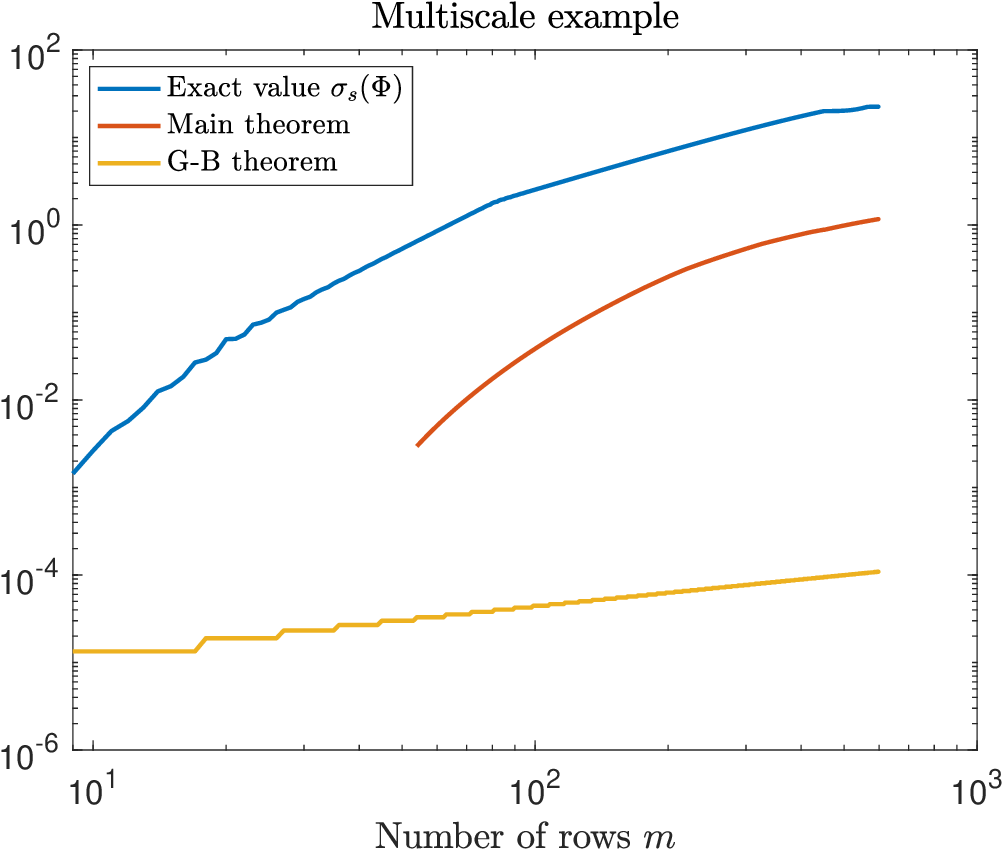}
	\end{subfigure}
	\caption{Left: Plot of $\calX$ defined in \eqref{eq:Xmotivation}. Right: Plot of $\sigma_s(\Phi(m,\calX))$ and two different lower bounds as functions of $m$ for $\calX$ defined in \eqref{eq:Xmotivation}. }
	\label{fig:motivational}
\end{figure}

In contrast, this experiment is for fixed $\calX$ and variable $m$ within a finite range. Determining $\sigma_s(\Phi(m,\calX))$ is naturally a discrete problem -- there are no large or small parameters to exploit. We are only aware of one prior work that applies to this example. It results from combining Gautschi \cite{gautschi1963inverses} and Baz\'an \cite{bazan2000conditioning} to obtain 
\begin{equation}
	\label{eq:classical}
	\sigma_s(\Phi(m,\calX))
	\geq \sqrt{ \frac {\lfloor \frac m s\rfloor} s} \, \min_{1\leq k\leq s} \Bigg\{\prod_{j\not=k} \frac{|e^{2\pi ix}-e^{2\pi i x_k}|} 2\Bigg\}. 
\end{equation}

We refer to this as the Gautschi-Baz\'an theorem, and will provide more details on its derivation in the comparisons section. Aside from this inequality, a main result of this paper, \cref{thm:main}, is applicable to this discrete problem, but with a mild restriction that $m\geq 6s$. The true $\sigma_s(\Phi(m,\calX))$, our main theorem, and the Gautschi-Baz\'an bound are displayed in Figure \ref{fig:motivational}. We see that our theorem offers a substantial improvement and better captures the true behavior. Just to highlight a dramatic improvement offered by \cref{thm:main}, when $m=400$, our main theorem underestimates the true value by a multiplicative factor of 21.0038, whereas the Gautschi-Baz\'an bound is off by a factor of 1.9687e+05.

\section{Main results}

\label{sec:main}

The goal of this paper is to provide explicit, interpretable, and accurate bounds for $\sigma_s(\Phi(m,\calX))$ for arbitrary $\calX$ when $\Delta(\calX)<\frac 1 m$. Doing so is a tricky balancing act. We require conditions on $\calX$ that are not too restrictive, yet are sufficiently informative enough that a resulting lower bound is not too loose. We avoid restrictive assumptions by working with general geometric notions. 
\begin{definition}
	Let $\tau \in (0,\frac 1 2]$ and $\calX\subset\T$ be a finite set. The $\tau$ {\it local sparsity of $\calX$} is
	$$
	\nu(\tau,\calX):=\max_{x\in\calX} \, \big|\{x'\in\calX \colon |x-x'|_\T \leq \tau \} \big| 
	\andspace \nu(\tau,\emptyset):=0. 
	$$
\end{definition} 

The $\tau$ local sparsity is the maximum number of elements in $\calX$ contained with a $\tau$ neighborhood of any $x\in \calX$. By definition, $\nu(\tau,\calX)=1$ if and only if $\Delta(\calX)>\tau$. Importantly, we have $\nu(\tau,\calU)\leq \nu(\tau,\calX)$ whenever $\calU\subset \calX$. If $\tau= \frac 1 2$, then $\nu(\tau,\calX)=|\calX|=s$, but it is possible that the local sparsity is significantly smaller than $s$. 

\begin{definition}
	For any $m\in \N_+$ and $\tau\in (0,\frac 1 2]$, we say a finite set $\calX\subset\T$ satisfies the $(m,\tau)$ {\it density criterion} if
	$$
	\frac{3\nu(\tau, \calX)} \tau \leq m.
	$$
\end{definition}

We call it the {density criterion} because $\frac {\nu(\tau,\calX)}{2\tau}$ can be interpreted as the density of $\calX$ at scale $2\tau$, so the assumption asserts that it cannot be bigger than $\frac{m}{6}$. This criterion is not difficult to fulfill for some $\tau$. Indeed, if we assume $m\geq 6s$ and select $\tau=\frac 1 2$, then $\frac{3\nu(\tau,\calX)}{\tau}=6s\leq m$, so the density criterion is satisfied. However, if $\calX$ satisfies the $(m,\tau)$ density criterion, then there may be infinitely many other $\tau$ that are also valid, and the choice of $\tau$ will influence the below estimates.  

We are almost ready to present our first main result. When interpreting the expressions in this paper, we use the standard convention that the product over an empty set is defined as 1. To simplify some of the notation that will appear in this paper, we define the subsets,
$$
\calB(x,\tau,\calX)
:=\{x'\in\calX\colon |x-x'|_\T\leq \tau\}, \andspace 
\calG(x,\tau,\calX):=\{x'\in\calX\colon |x-x'|_\T> \tau\}.
$$ 
We will refer to these as the ``bad" and ``good" sets respectively, and this terminology will make sense later. We define a special function $\phi\colon [1,\infty)\to [1,\infty)$ by
\begin{equation}
	\label{eq:phidef}
	\phi(t):=\frac t {\lfloor t \rfloor}.
\end{equation}
This function appears in several bounds since our methods depend on number theoretic properties of several quantities. Note that $\phi(t)\leq \min\{2,1+\frac 1 {t-1}\}$, so in particular, $\phi(t)\to 1$ as $t\to\infty$. Terms that involve $\phi(t)$ are inconsequential when $t$ is large.

\begin{theorem}
	\label{thm:main}
	Let $m,s\in\N_+$ such that $s\geq 2$ and $m\geq 6s$. Suppose $\calX=\{x_k\}_{k=1}^s\subset\T$ and pick any $\tau$ such that $\calX$ satisfies the $(m,\tau)$ density criterion. For each $k=1,2,\dots,s$, define 
	$$
	\calG_k:=\calG(x_k,\tau,\calX), \quad \alpha_k:=\frac{ |\calB(x_k,\tau,\calX)|}{2m  - 4\nu(\tau,\calG_k)\tau^{-1}}, 
	$$ 
	and the subsets 
	$$\calI_k:=\{x\in\calX\colon 0<|x-x_k|\leq \alpha_k\} \andspace \calJ_k:=\{x\in\calX\colon  \alpha_k <|x-x_k|\leq \tau \}.$$ 
	Then we have
	\begin{equation}
		\label{eq:main1}
		\frac{1}{\sigma_s^2(\Phi(m,\calX))} 
		\leq \sum_{k=1}^s 2^{\nu(\tau,\calG_k)} \frac{2\alpha_k }{1-2\alpha_k} \prod_{x\in \calJ_k} \phi\(\frac 1 {2|x-x_k|_\T}\)^2  
		\prod_{x\in\calI_k} \frac{\alpha_k^2}{(1-2\alpha_k)^2|x-x_k|_\T^2},
	\end{equation}
	and in particular, 
	\begin{equation}
		\label{eq:main2}
		{\sigma_s(\Phi(m,\calX))}
		\geq \min_{1\leq k\leq s} \  \Bigg\{ \frac 1 {\sqrt{4s \alpha_k}} \frac{1}{\sqrt{2^{\nu(\tau,\calG_k)}}} \frac{1}{2^{|\calJ_k|}}
		\prod_{x\in\calI_k} \frac{|x-x_k|_\T} {2\alpha_k} \Bigg\}.
	\end{equation}
\end{theorem}

Both inequalities in this theorem provide ``multiscale" lower bounds for $\sigma_s(\Phi)$. Let us explain what this terminology means in the context of \eqref{eq:main2}, which is a simplified version of \eqref{eq:main1}. Here, we fix a reference point $x_k\in\calX$. The ranges $(\tau,\frac 12]$, $(\alpha_k,\tau]$, and $(0,\alpha_k]$ consist of the coarse, intermediate, and small scales respectively, and the elements in $\calX$ whose distance to $x_k$ lie in these three ranges are $\calG_k$, $\calJ_k$, and $\calI_k$, respectively. An example is shown in Figure \ref{fig:threescales}. 
\begin{enumerate}[(a)] \itemsep-2pt
	\item 
	The set $\calG_k$ contributes a factor of $2^{-\nu(\tau,\calG_k)/2}$ to the smallest singular value. Note that $\nu(\tau,\calG_k)$ may be significantly smaller than $s-1$ for many types of $\calX$, including the motivational example. Hence, points in $\calX$ that are $\tau$ away from $x_k$ typically have little influence on the lower bound.
	\item 
	The set $\calJ_k$ contributes a factor of $2^{-|\calJ_k|}$, so each element in $\calJ_k$ contributes a multiplicative factor of $\frac 12$ to the lower bound. Note that $\alpha_k\leq \frac  \tau 2$ due to the $(m,\tau)$ density criterion, so $\alpha_k$ is naturally the next smaller scale following $\tau$. 
	\item 
	The set $\calI_k$ has the greatest amount of influence on the lower bound. Notice that each term inside the product in \eqref{eq:main2} is at most $\frac 1 2$, but may be significantly smaller depending on the structure of $\calX$ near $x_k$. 
	For instance, if we let $s_{k,\ell}\in \N$ denote the number of elements in $\calX$ whose distance to $x_k$ lies in $(2^{-\ell-1}\alpha_k,2^{-\ell}\alpha_k]$, then
	$$
	\log_2 \( \prod_{x\in\calI_k} \frac{|x-x_k|_\T} {2\alpha_k} \)
	\asymp -\sum_{\ell=0}^\infty s_{k,\ell}\, \ell. 
	$$	
	This illustrates that the product term is not equally influenced by all distances less than $\alpha_k$ and that it actually depends on the structure of $\calX$ near $x_k$ at infinitely many finer scales. 
\end{enumerate}

\begin{figure}[h]
	\centering
	\begin{tikzpicture}
		[scale=1.1,
		place/.style={circle,draw=blue!50,fill=blue!20,thick},
		place2/.style={diamond,fill=blue,inner sep = 0.6mm},
		place4/.style={rectangle,fill=orange,inner sep = 0.9mm},
		place3/.style={regular polygon, regular polygon sides=3,fill=red,thick,inner sep = 0.4mm}]
		\draw[thick] (-7,0) -- (7,0);
		\node at (0,0) [circle,fill=black,inner sep = 0.7mm] {};
		\node[below] at (0,-0.2) {$x_k$};
		\node at (1,0) [place3] {};
		\node at (.25,0) [place3] {};
		\node at (-0.5,0) [place3] {};
		\node at (2,0) [place4] {};
		\node at (2.25,0) [place4] {};
		\node at (-3,0) [place4] {};		
		\node at (2.5,0) [place4] {};
		\node at (-5,0) [place2] {};
		\node at (-5.5,0) [place2] {};
		\node at (-5.25,0) [place2] {};
		\node at (4,0) [place2] {};	
		\node at (-3.5,0) [place2] {};		
		\draw[red,thick,<->] (-1,0.4) -- (1,0.4);
		\draw[red,thick](0,0.3) -- (0,0.5);
		\node[red,above] at (0,0.45) {$2\alpha_k$};
		\draw[orange,thick,<->] (-3,1.3) -- (3,1.3);
		\draw[orange,thick](0,1.2) -- (0,1.4);
		\node[orange,above] at (0,1.4) {$2\tau$};
		\node at (-6.4,2) [place3] {};
		\node[right,color=red] at (-6.2,2) {$\calI_k$};
		\node at (-6.4,1.5) [place4] {};
		\node[right,color=orange] at (-6.2,1.5) {$\calJ_k$};
		\node at (-6.4,1) [place2] {};
		\node[right,color=blue] at (-6.2,1) {$\calG_k$};
		\draw[thick] (-6.75,.75) -- (-6.75,2.25) -- (-5.3,2.25) -- (-5.3,.75)--(-6.75,.75);
	\end{tikzpicture}
	\caption{An example of the three sets $\calI_k$, $\calJ_k$ and $\calG_k$ defined in \cref{thm:main}.}
	\label{fig:threescales}
\end{figure}

If there is a $\tau\ll \frac 1 2$ for which the density criterion holds, then this theorem effectively communicates a localization phenomenon. Even though the Fourier transform is non-local, in the sense that all elements of $\calX$ participate, only those whose distances are closer than $\tau$ substantially contribute. On the other hand, if $\tau=\frac 1 2$ is selected, then there is no localization.  

Motivated by inverse problems where only weak information about $\calX$ is known or can be reasonably assumed, we provide a different lower bound for $\sigma_s(\Phi)$ in terms of any lower bound for $\Delta(\calX)$.  The following is our second main result.
\begin{theorem}
	\label{thm:main2}	
	Let $m,s\in\N_+$ such that $s\geq 2$ and $m\geq 6s$, and let $\delta \in (0,\frac 1 m]$. Suppose $\calX\subset\T$ is a set of cardinality $s$ with $\Delta(\calX)\geq \delta$, and pick any $\tau$ such that $\calX$ satisfies the $(m,\tau)$ density criterion. For each $k=1,2,\dots,s$, define 
	\begin{equation*}
		\calG_k:=\calG(x_k,\tau,\calX), \quad r_k:=|\calB(x_k,\tau,\calX)|, \andspace n_k:= \left\lfloor m- \frac{2\nu(\tau,\calG_k)} \tau \right\rfloor. 
	\end{equation*}
	Then we have
	\begin{equation}
		\label{eq:main3}
		\frac{1}{\sigma_s^2(\Phi(m,\calX))}
		\leq \frac{4e^2}{\pi^2} \sum_{k=1}^s 2^{\nu(\tau,\calG_k)}  \frac {\phi(\frac {n_k} {r_k})}{r_k n_k} \( \frac{2e \phi(\frac {n_k} {r_k})}{\sin(\frac \pi 2 n_k\delta)} \)^{2r_k-2},
	\end{equation}
	and in particular, 
	\begin{equation}
		\label{eq:main4}
		\sigma_s(\Phi(m,\calX))
		\geq \frac{\pi}{2 e} \sqrt{\frac{m}{6s}} \, \min_{1\leq k\leq s} \, \Bigg\{ \sqrt{\frac{r_k}{2^{\nu(\tau,\calG_k)}}} \( \frac{m \delta}{12 e} \)^{r_k-1}\Bigg\}.
	\end{equation}
\end{theorem}

We emphasize that $\delta$ is an independent parameter, so the theorem is applicable to sets for which $\Delta(\calX)$ is arbitrarily small. This theorem is written from the perspective of $\delta$. In \eqref{eq:main4}, the exponent on $\delta$ is $r_k-1$, which shows that interactions between $x$ and $x_k$ at scales smaller than $\tau$ are most significant. This theorem assumes that $\delta\leq \frac 1 m$, which can be relaxed by adapting this theorem's proof, but with some additional technical complications. This is not a prohibitive assumption since the estimate \eqref{eq:wellseparated} can be used whenever $\delta>\frac 1 m$. 

Compared to \cref{thm:main}, \cref{thm:main2} is easier to employ since it requires less information about $\calX$, but it generally yields a looser bound. This is expected since \cref{thm:main} contains product terms that depend on the pairwise distances between elements, whereas \cref{thm:main2} has effectively replaced all of these distances by $\delta$. Both theorems give similar predictions if all small scales are approximately $\delta$. For sets with many scales between $\delta$ and $\tau$, it is generally advisable to use \cref{thm:main} instead.  

To use \cref{thm:main,thm:main2}, one first needs to select an appropriate $\tau$ for which the $(m,\tau)$ density criteria holds, and as mentioned earlier, there may be infinitely many choices. For certain sets, one could select $\tau$ heuristically based on $m$ and properties of $\calX$. Another option is to sweep through a collection $\calT$ of $\tau$ and pick a $\tau$ that maximizes whichever bound one would like to use. We will discuss the computational cost at the end of this article.

Clumps models for $\calX$ were independently introduced in \cite{li2021stable,batenkov2020conditioning} and were used to control the condition number of tall Fourier matrices. There are some subtle differences between the definitions in these papers, so to facilitate the presentation and to avoid giving two separate definitions, we work with the following boarder definition that encapsulates both frameworks. For sets $\calU,\calV\subset \T$, we define the diameter and distance,
$$
\text{diam}(\calU):=\sup_{u,u'\in \calU} |u-u'|_\T \andspace
\text{dist}(\calU,\calV):=\inf_{u\in\calU,v\in\calV}|u-v|_\T.
$$

\begin{definition}
	\label{def:clumpsdecomp}
	A set $\calX\subset\T$ consists of {\it separated clumps} with parameters $(s,\delta,r,\lambda,\alpha,\beta)$ if the following hold. We have $|\calX|=s$, $\Delta(\calX)\geq \delta$, and there is a disjoint union 
	\begin{equation*}
		\calX = \calC_1\cup \calC_2 \cup \cdots \cup \calC_r,
	\end{equation*}
	where each $\calC_k$ is called a {\it clump} such that
	$$
	\max_{1\leq k\leq r} |\calC_k|= \lambda, \quad \max_{1\leq k\leq r} \text{diam}(\calC_k)\leq \alpha, \andspace \min_{j\not=k} \dist(\calC_j,\calC_k)> \beta>\alpha \quad \text{if $r\geq 2$}.
	$$ 
\end{definition} 

A few remarks are in order. For a fixed $\calX$, the choice of parameters is not unique and it is usually advisable to select valid parameters that minimize $\lambda$. If $r=1$, then $s=\lambda$ and $\beta$ is not a meaningful parameter since there is only a single clump. This is why the clump separation requirement is necessary only when $r\geq 2$. Notice $\beta>\alpha$ is included in the assumption so that distances between clumps exceeds within a clump. We will see that $\lambda$ plays the role of $\max_{1\leq k\leq s} r_k$, where $r_k$ was defined in \cref{thm:main2}. 

There are natural situations where a set consisting of separated clumps also satisfies the requirements of our main results, as shown in the next proposition. 

\begin{corollary}
	\label{prop:mainclumps}
	Let $m,s\in \N_+$ such that $s\geq 2$ and $m\geq 6s$. Suppose $\calX$ consists of separated clumps with parameters $(s,r,\delta,\lambda,\alpha,\beta)$ such that $\delta\leq \frac 1 m$. If $r>1$, also assume that $\beta\geq \frac {3 \lambda} m$. Set $\tau = \frac 1 2$ if $r=1$, otherwise let $\tau=\beta$. Then $\calX$ satisfies the $(m,\tau)$ density criterion, and the conclusions of Theorems \ref{thm:main} and \ref{thm:main2} hold. In particular, we have 
	$$
	\sigma_s(\Phi(m,\calX))
	\geq \frac{\pi}{2 e} \sqrt{\frac{m}{12s}} \( \frac{m \delta}{12 \sqrt 2 e} \)^{\lambda-1} .
	$$
\end{corollary}

The condition that $\beta$ scales linearly in $\frac \lambda m$ is the best one can expect without imposing further restrictions on $\calX$. Indeed, if we allow $\beta < \frac \lambda m$, then it may occur that $m<s$. Although \cref{prop:mainclumps} provides the same or worse estimate compared to \cref{thm:main2}, we included the corollary in order to compare with prior results for clumps.

While this paper focuses on the smallest singular value, the techniques developed in this paper provide a straightforward and nontrivial upper bound for the largest singular value. 

\begin{theorem}
	\label{thm:upper}
	Let $m,s\in\N_+$ such that $m\geq s$. For any $\calX\subset\T$ of cardinality $s$ and $\tau\in (0,\frac 1 2]$ such that $m>\frac 1 \tau$, we have 
	$$
	\sigma_1(\Phi(m,\calX))
	\leq \sqrt{\nu(\tau,\calX)\(m+\frac 1 \tau\)}.
	$$
\end{theorem}

For comparison purposes, recall the trivial bound $\sigma_1(\Phi)\leq \|\Phi\|_F= \sqrt{ms}$. Observe that \cref{thm:upper} provides a significantly better upper bound if the $\tau$ local sparsity of $\calX$ is much smaller than $s$. For example, if we were to apply the above theorem for $\tau=\frac 2 m$, then $\sigma_1(\Phi)=\sqrt{3 \nu(\tau,\calX)m/2}$, which is an improvement over the trivial bound if $\nu(\tau,\calX)< \frac {2s} 3$. 

\section*{Organization}

\cref{sec:comparison} provides detailed comparisons with prior work on the condition number of Fourier matrices, and serves as an expanded version of \cref{sec:background}. There, we will see that our main theorems capture the scaling and localization phenomena that are missing from the classical Gautschi-Baz\'an theorem. In the case of clumps, we will see that \cref{prop:mainclumps} is equivalent to the lower bound in \eqref{eq:clumps} modulo universal constants, while holding under tremendously weaker separation assumptions. 

\cref{sec:numerics} is dedicated to examples and numerical simulations, with comparisons to the predictions provided by this paper. It also provides more details regarding the motivational example in \cref{sec:motivation}. There, we provide some extreme examples that illustrate when localization does (not) occur. One on hand, there are examples where $\tau=\frac C m$ like in the clumps model, while in other examples, $\tau=\frac 1 2$ such as for sparse spike trains. They illustrate the effectiveness and flexibility of the main results.  

The remaining portions deal with proofs. \cref{sec:proofstrat} develops the main tool called the polynomial method and introduces two specific trigonometric interpolation problems that are connected to the main theorems. This section also outlines the main strategy for proving the main theorems without technical details. \cref{sec:interpolation} addresses the ``good" and ``bad" interpolation problems, and how the resulting polynomials are related to other interpolation strategies. \cref{sec:proofs} contains proofs of the main results stated in \cref{sec:main}.

\section{Comparison with prior art}

\label{sec:comparison}

\subsection{Comparison to classical estimates}

Classical versus modern papers on Fourier matrices centers on the differences between square versus rectangular. A $s\times s$ Fourier matrix is perfectly conditioned if and only if $\calX$ is some shift of the uniform lattice $\{\frac k s\}_{k=1}^s$, see \cite{berman2007perfect}. It is natural to wonder whether it is possible to relax both sides of this characterization. It would be a delicate task, since \cite{cordova1990vandermonde} established that if $\calX$ consists of the first $s$ terms of the Van Der Corput sequence, then $\kappa(\Phi(s,\calX))=1$ only if $\log_2(s)$ is an integer, but grows like $\sqrt s$ otherwise. This example illustrates that it is possible for $\kappa(\Phi(s,\calX))$ to be unbounded in $s$ even if $\calX$ is ``spread out" in $\T$. Stability of the discrete Fourier transform matrix to perturbations of its nodes via the Kadec-$\frac 14$ theorem were derived in in \cite{yu2023stability}.

The results listed in the previous paragraph illustrate the brittleness of square matrices, while rectangular ones are much more robust. Any $m\times s$ sub-matrix of the $m\times m$ discrete Fourier transform matrix is perfectly conditioned even though the nodes are not uniformly spaced on the circle. More generally, notice from inequality \eqref{eq:wellseparated} that $\Delta(\calX)>\frac 1 m$ implies the conditioning of $\Phi(m,\calX)$ can be bounded uniformly in both $m$ and $s$. It is important to mention that this inequality only applies to rectangular matrices since $\Delta(\calX)>\frac 1 m$ and $|\calX|=s$ imply that $m>s$. These observations should be compared with the ones listed in the previous paragraph for square matrices.

Results for square matrices can be used to deduce bounds for rectangular ones, beyond the trivial relationship $\sigma_s(\Phi(m,\calX))\geq \sigma_s(\Phi(s,\calX))$. We first start with Gautschi \cite[Theorem 1]{gautschi1963inverses} for square matrices, 
$$
\|\Phi(s,\calX)^{-1}\|_\infty
\leq \max_{1\leq k\leq s} \Bigg\{\prod_{j=1,\, j\not=k}^s \frac{2}{|e^{2\pi ix_j}-e^{2\pi i x_k}|} \Bigg\}.
$$
Here, $\|\cdot\|_p$ denotes the $\ell^p\to\ell^p$ operator norm. Next, Baz\'an \cite[Theorem 1]{bazan2000conditioning} showed that whenever $\frac m s \in \N$, then 
$$
\|\Phi(m,\calX)^\dagger\|_2\leq \sqrt{\frac s m} \, \|\Phi(s,\calX)^{-1}\|_2.
$$
Combining the above two inequalities, that $\|A\|_2\leq {\sqrt s}\|A\|_\infty$ if $A\in \C^{s\times s}$, and $\sigma_s(\Phi(m,\calX))\geq \sigma_s(\Phi(\lfloor \frac m s \rfloor s,\calX))$, we obtain the Gautschi-Baz\'an theorem, which was stated in inequality \eqref{eq:classical}.  

Comparing the Gautschi-Baz\'an theorem with \cref{thm:main}, we see that there are two main differences. First, the former does not exhibit {\it localization} since the product in \eqref{eq:classical} is taken over all $x\in \calX\setminus\{x_k\}$, whereas in the latter, the product is over all $x\in\calI_k$, while the further away elements are less significant. Note that if $|x-x_k|_\T$ is small, then $|e^{2\pi ix}-e^{2\pi i x_k}|$ is comparable to $|x-x_k|_\T$. Second, the former does not exhibit the correct {\it scaling} in front of $|e^{2\pi ix}-e^{2\pi i x_k}|$. In the latter, notice that each term has a helpful $\alpha_k$ factor, which could be on the order of $\frac 1 m$ depending on $\calX$. The localization and scaling phenomenon manifest when we consider rectangular Fourier matrices, which are absent for square ones. Examining the proof of \cite[Theorem 1]{bazan2000conditioning}, we see that is treats tall Fourier matrices as $\lfloor \frac m s \rfloor$ independent $s\times s$ blocks, and does not fully exploit the algebraic structure of tall Fourier matrices.

Continuing the remarks made in the previous paragraph, there is an elementary explanation for why tall Fourier matrices should behave differently from square ones. Notice that $(\Phi^* \Phi)_{j,k} =D_m(x_j-x_k),$ where $D_m(t):=\sum_{k=0}^{m-1} e^{2\pi i kt}$ is the Dirichlet kernel. We easily see that $|D_m(t)|$ is on the order of $m$ on the interval $[-\frac 1 m , \frac 1 m]$ and decays at a rate of $|t|^{-1}$ away from $0$. This means that the Gram matrix $\Phi^*\Phi$, for fixed $\calX$ and increasing $m$, becomes increasingly diagonally dominant. Basic and generic tools such as the Gershgorin circle theorem fail to provide any meaningful results when $\Delta(\calX)<\frac 1 m$ and $s\geq 3$ because the diagonal entries of $\Phi^*\Phi$ are $m$ while the $\ell^1$ norm of off-diagonal rows and columns of $\Phi^*\Phi$ exceed $m$. Instead, the proof methods used in this paper specifically take advantage of the algebraic structure of Fourier matrices and are able to obtain finer results. 

\subsection{Comparison to clumps}

In this part, we compare \cref{prop:mainclumps} with the results in \cite{li2021stable,batenkov2020conditioning,kunis2020smallest,batenkov2021single}. As usual, we let $m$ and $s$ denote the number of rows and columns of $\Phi$. We will only compare the general scaling of the model parameters and do not compare universal constants, since the latter can be improved by optimizing their proofs or by providing more accurate but complicated expressions. When comparing our main theorems with other papers, we will generally ignore distinctions between $m-1$, $m$, and $2m$, since the extraneous factors can be absorbed into other constants. 

The result \cite[Theorem 2.7]{li2021stable} shows that if $\calX$ consists of separated clumps with parameters $(s,\delta,r,\lambda,\frac 1 m,\beta)$ such that
\begin{equation}
	\label{eq:clumpsep1}
	m\geq s^2 \andspace \beta \geq \frac{20}{m-1} \sqrt{\frac{s\lambda^5}{(m-1)\delta}} \quad \text{if}\quad r>1,
\end{equation}
then there exist explicit universal constants $C>0$ and $c\in (0,1)$ such that 
$$
\sigma_s(\Phi(m,\calX))
\geq C \sqrt{\frac{m-1}{\lambda}} \big(c(m-1)\delta \big)^{\lambda-1}. 
$$
One main drawback of condition \eqref{eq:clumpsep1} is that $\beta\to\infty$ as $\delta\to 0$, so for sufficiently small $\delta$, the theorem only applies when there is only a single clump. Some improvements to the explicit constants and variations of this inequality can be found in \cite{kunis2020smallest}. All results in this paper also require separation conditions for which $\beta\to\infty$ as $\delta\to0$. 

\cref{prop:mainclumps} shows that under the same hypotheses \eqref{eq:clumpsep1}, this paper's main results are applicable and they yield the same $C\sqrt m (cm \delta)^{\lambda-1}$ estimate with different constants. However, \cref{prop:mainclumps} requires significantly weaker clump separation assumptions and relationship between $m$ versus $s$. Importantly, it removes the artificial behavior that $\beta$ explodes in the limit that $\delta$ goes to zero.

Moving on, \cite[Theorem 2.2]{batenkov2021single} shows that if $\calX$ consists of separated clumps with parameters $(s,\delta,r,\lambda,\alpha,\beta)$ such that 
\begin{equation}
	\label{eq:clumpsep2}
	\frac {c_1(\lambda)s}{\pi \beta} \leq m \leq \frac{c_2(\lambda)}{\pi s\alpha},
\end{equation}
for some $c_1(\lambda),c_2(\lambda)>0$ depending only on $\lambda$, then there is an explicit universal constant $C'>0$ such that
$$
\sigma_s\big(\Phi(2m+1,\calX)\big)
\geq C' \sqrt{m} \, \( \frac{m\delta}{16 e}\)^{\lambda-1}.
$$
Some of the constants in these expression are different than those in \cite{batenkov2021single} since that paper identifies $\T$ with $[-\pi,\pi)$ as opposed to $[0,1)$ in this paper. 

Although $c_1(\lambda)$ and $c_2(\lambda)$ are not given explicitly, \cite[Section 6.3]{batenkov2021single} shows that $c_1(\lambda)> C(1+\lambda^9)$ for some universal $C>0$ and $c_2(\lambda)\leq 2\pi$. Hence, condition \eqref{eq:clumpsep2} requires, for all sufficiently large $\lambda$,
$$
\alpha\leq \frac{2}{sm} \andspace \beta\geq \frac {C(1+\lambda^9)s} m\geq \frac {3\lambda }m.
$$
This establishes that \cref{prop:mainclumps} provides a similar lower bound, but again, under significantly weaker assumptions.

Clumps models were also introduced in \cite{batenkov2020conditioning} to bound the smallest singular value of a ``continuous" analogue, whereby $\Phi$ is replaced with an integral operator. In particular, \cite[Corollary 3.6]{batenkov2020conditioning} assumes that $\calX$ consists of separated clumps and with the additional requirement that $\diam(\calX)\leq \frac{1}{\pi s^2}$. It is not possible rescale this result to avoid this requirement, so we cannot provide a reasonable comparison. Nevertheless, the restriction that $\calX$ is contained in an interval of length $\frac{1}{\pi s^2}$ is removed in a follow-up result \cite{batenkov2021single}, which we already compared to. 

\subsection{Other related work}

The ``colliding nodes" model, where $\calX$ can be decomposed into clumps where each one has exactly two elements, was studied in \cite{kunis2021condition}. This is much more restrictive than the clumps model and can be treated with specialized tools that cannot be extend to more complicated and general sets. 

There is a plethora of papers that examine sub-matrices of the discrete Fourier transform matrix, see \cite{barnett2022exponentially} and references therein. This would correspond to the situation where $\calX\subset\{\frac k n\}_{k=1}^n$ and $n$ is a large parameter that can be selected independent of $m,s$. This setting is more specialized since there are cancellation properties and explicit formulas that are not available in the general case.

\section{Numerical simulations and examples}
\label{sec:numerics}

\subsection{Setup and definitions}

When comparing the true value of $\sigma_s(\Phi)$ and our estimated one, we use our more accurate estimate \eqref{eq:main1} from \cref{thm:main}. The software that reproduces the figures in this paper are publicly available on the author's Github repository \footnote{\href{https://github.com/weilinlimath}{https://github.com/weilinlimath}}, which is also linked to the author's personal website \footnote{\href{https://weilinli.ccny.cuny.edu}{https://weilinli.ccny.cuny.edu}}.

The behavior of Fourier matrices was numerically evaluated in \cite{li2021stable,batenkov2021single,batenkov2020conditioning,kunis2020smallest} under the {\it super-resolution limit}, whereby $m$ is sufficiently large and there is a family of $\{\calX_m\}_{m=1}^\infty$ for which $\Delta(\calX_m)\to 0$, or alternatively, $m\to\infty$ and $\Delta(\calX_m)\to 0$ with some relationship between $m$ and $\Delta(\calX_m)$. This is an important scaling in the theory of super-resolution and the behavior of $\sigma_s(\Phi)$ greatly simplifies in this scenario. The main results of this paper can be used for the super-resolution limit as well and would give equivalent predictions up to implicit constants, see \cref{prop:mainclumps}. One can consider a complementary scaling, called the {\it well-separated case}, whereby $\calX$ is fixed and $m\to\infty$. In this case, \eqref{eq:wellseparated} is applicable.

Rather than look at either scaling again, we look at more challenging examples. In the absence of a large or small parameter, $\sigma_s(\Phi)$ is naturally a discrete quantity. Nonetheless, even though our main theorem is proved using analytic tools, it only requires a weak assumption that $m\geq 6s$, so it is applicable to a greater variety of examples. We are only aware of one other result with this generality, which is the Gautschi-Baz\'an theorem in \eqref{eq:classical}. 

Since we provide lower bounds for the smallest singular value, it makes sense to quantify the quality of approximation by a multiplicative factor. That is, we define the
$$
\text{inaccuracy factor} := \frac{\text{true value}}{\text{estimated value}}.
$$
Of course, this quantity is lower bounded by $1$.   

\subsection{The motivational example revisited}

Here, we provide additional details for the motivational example in \cref{sec:motivation}. First, notice that for a fixed $\calX$, the set of $\tau$ for which the $(m,\tau)$ density criterion is satisfied are nested increasing sets as $m$ increases. More precisely, if we define 
$$
\calS(m,\calX):=\{\tau \colon \calX \text{ satisfies the } (m,\tau) \text{ density criterion}\},
$$
then $\calS(m,\calX)\subset\calS(m+1,\calX)$. So as $m$ increases, we have the option of choosing $\tau$ smaller in order to reduce the number $x\in\calX$ that are close to each $x_k$. However, we do not simply define $\tau$ as the infimum of $\calS(m,\calX)$ because $\alpha_k$ may increase when $\tau$ decreases. Choosing an optimal $\tau$ is beyond the scope of this paper. It is not difficult to select reasonable a candidate based on intuition, or trial and error. 

\begin{figure}[h]
	\centering 
	\includegraphics[width=0.49\textwidth]{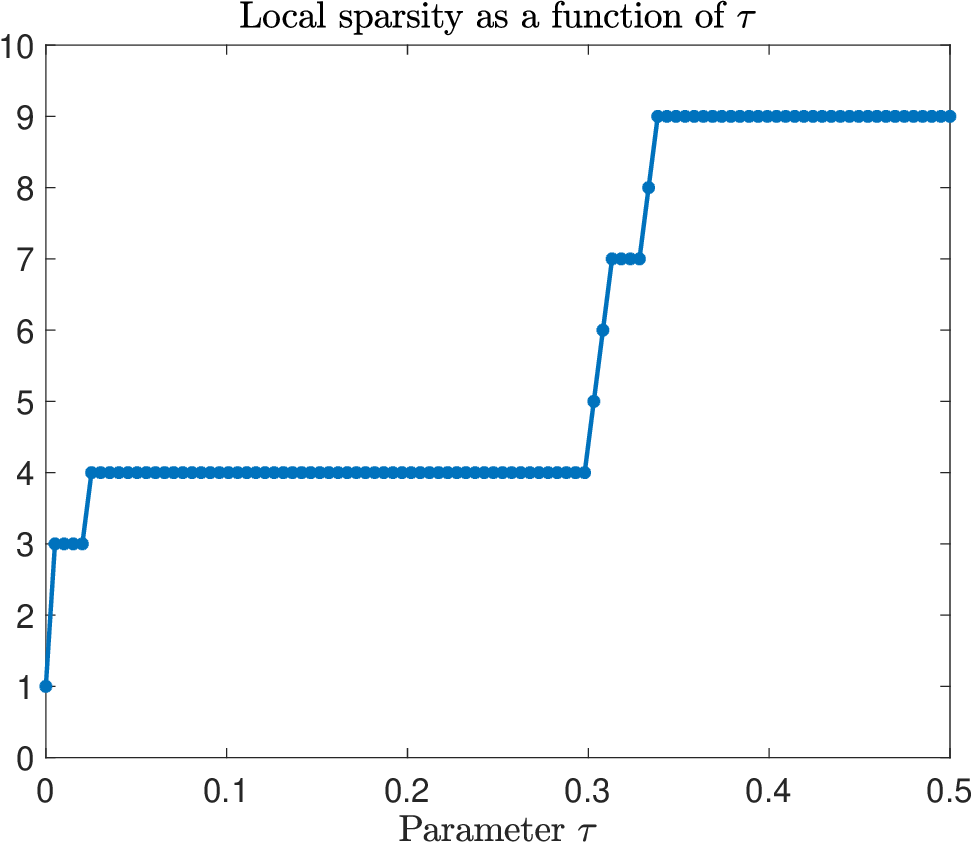}
	\caption{Plot of $\nu(\tau,\calX)$ as a function of $\tau\in (0,\frac 1 2]$. }
	\label{fig:increasingM}
\end{figure}

Returning back to the motivational example, after some calculations, we define 
$$
\tau:=
\begin{cases} 
	\ \frac 2 {30} &\text{if } m\in (450,600], \\
	\ \frac 3 {10} &\text{if } m \in [54,450]. 
\end{cases}	
$$
For these corresponding values of $\tau$, it can be easily checked that $\nu(\tau,\calX)$ is 3 or 4 and that $\calX$ satisfies the $(m,\tau)$ density criterion. To visualize the former, we have plotted $\nu(\tau,\calX)$ as a function of $\tau$ in Figure \ref{fig:increasingM}. Note that our choices for $\tau$ are not optimized, but were chosen according to reasonable heuristics.  

As shown in Figure \ref{fig:motivational}, our theorem yields a significantly more accurate prediction, which becomes more apparent as $m$ increases. This occurs because the effective scale $\tau$ should be chosen to decrease in $m$ and the distance between nearby elements is scaled according to $\alpha_k$, neither of which are captured in the Gautschi-Baz\'an theorem. Additionally, the results in \cite{li2021stable,kunis2020smallest} are not applicable for any $m\in [9,600]\cap\N$, because the separation condition \eqref{eq:clumpsep1} is not fulfilled, while \cite{batenkov2020conditioning} cannot be used since $\calX$ is not contained in an interval of length $\frac 1 {\pi s^2}$, and it is unclear whether \cite{batenkov2021single,batenkov2020conditioning} can be used since they contain implicit constants in their separation criterion. 

\subsection{Another multiscale example} 

Unlike the motivational example in \cref{sec:motivation} where $\calX$ was fixed and $m$ varies, we consider the reverse situation where $m$ is fixed and we have a family of sets $\calX_\epsilon$ parameterized by a $\epsilon\in (0,1]$. Consider the set 
\begin{gather}
	\begin{split}
	\calX_\epsilon&:=\calX_{1,\epsilon} \cup \calX_2(\epsilon) \cup  \calX_{3,\epsilon}, \wherespace\\
	\calX_{1,\epsilon} = \epsilon\{0,\tfrac 1 {90}, \tfrac 2{90}, \tfrac 3{90}\}, \quad
	&\calX_{2,\epsilon} = \tfrac 1 3 + \epsilon \{0,\tfrac 1 {200}, \tfrac 2  {200}\}, \andspace
	\calX_{3,\epsilon} = \tfrac 2 3 + \epsilon \{0, \tfrac 1 {500}\}. 
	\end{split}
	\label{eq:multiplescales}
\end{gather}
We have defined $\calX_\epsilon$ in this way to emphasize that while $\epsilon$ controls the minimum separation since $\Delta(\calX_\epsilon)=\frac{\epsilon}{500}$, the three sets $\calX_1$, $\calX_2$, and $\calX_3$ are still of different scales for each $\epsilon$. 

Since $\Delta(\calX_\epsilon)\leq \frac 1{500}$ for any $\epsilon$, we consider only $m\leq 500$. If we pick $\tau=\frac 3 {10}$, then $\nu(\tau,\calX_\epsilon)=4$ for all $\epsilon$. For two separate experiments, we select $m=400$ and $m=100$. Note that $\calX_\epsilon$ satisfies the $(m,\tau)$ density criterion for all values of $\epsilon$.

\begin{figure}[h]
	\centering
	\includegraphics[width=0.49\textwidth]{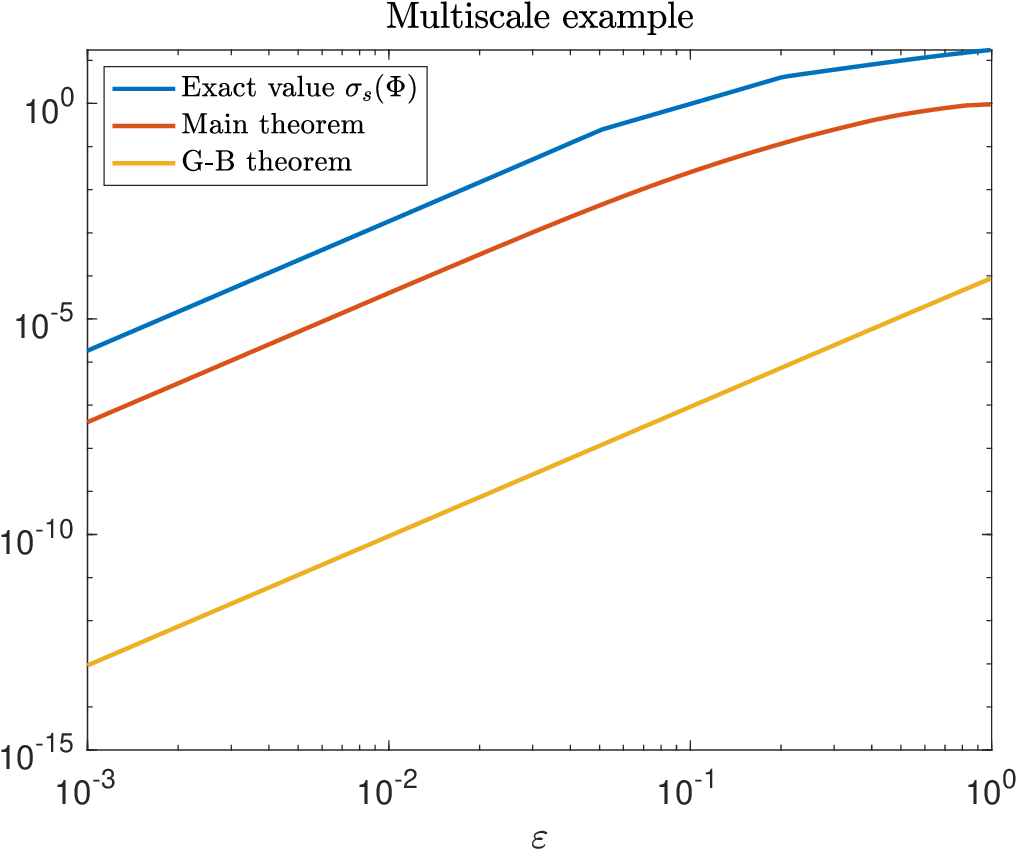}
	\includegraphics[width=0.49\textwidth]{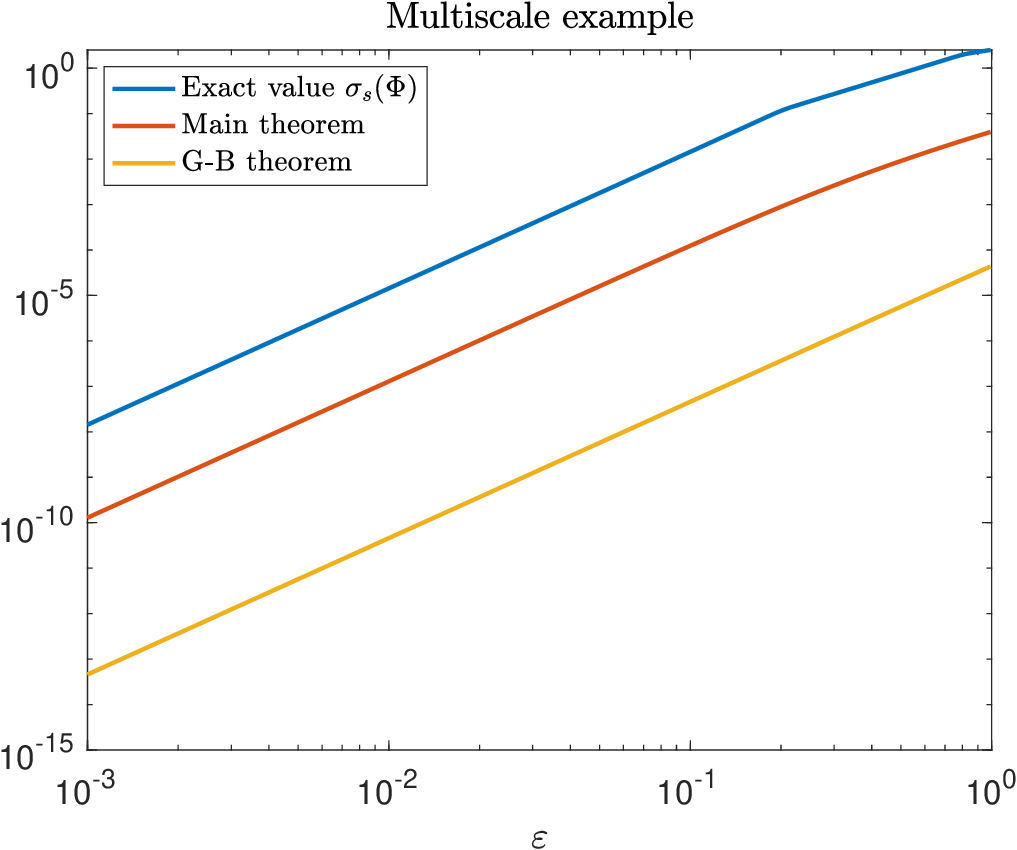}
	\caption{For $\calX$ defined in \eqref{eq:multiplescales}, plot of $\sigma_s(\Phi(m,\calX_\epsilon))$, main theorem, and Gautschi-Baz\'an bound, as a function of $\epsilon$, with $m=400$ on the left and $m=100$ on the right.}
	\label{fig:multclumps}
\end{figure}

The results are shown in Figure \ref{fig:multclumps}. We see from the simulations that our lower bound matches the true behavior of the smallest singular value. Notice that for both experiments, $\sigma_s(\Phi)$ is piece-wise linear, which is expected. Indeed, our theory states that the only significant interactions between $x,x'\in\calX$ are those for which $|x-x'|\leq \tau$. For $\tau=\frac 3 {10}$, the sets $\calX_1$, $\calX_2$, and $\calX_3$ do not have significant interactions due to our choice of $\tau$. They also have cardinality 4, 3, and 2 respectively for all $\epsilon$, and the interactions between elements in each $\calX_k$ scales linearly with $\epsilon$. Hence, according to \cref{thm:main}, we expect 
$$
\sigma_s(\Phi(m,\calX_\epsilon))
\gtrsim c_1 \epsilon^3 + c_2\epsilon^2+c_3\epsilon^1,
$$
for some universal constants that can be explicitly computed. Hence, as $\epsilon$ varies, depending on the regime of $\epsilon$ and the size of $c_1,c_2,c_3$, the dominant term in this inequality changes. In fact, Figure \ref{fig:multclumps} shows that $\sigma_s(\Phi)$ appears to consists of three power-law pieces; on a log-log plot, they have slopes approximately $3.0055$, $2.0321$, and $0.9485$, which is consistent with our prediction.

\subsection{Sparse spike train}

In this example, we consider an extreme situation where $\tau$ cannot be chosen on the order of $\frac 1 m$ even though the number of elements in an interval of length $\frac 2 m$ is at most 3. For any $\epsilon\in (0,1]$ and $s\in [5,30]\cap \N_+$, we set $m=200$ and consider the following set,
\begin{equation}
	\label{eq:spiketrain}
	\calX_{s,\epsilon}= \epsilon \left\{0, \, \frac{1}{m}, \, \dots, \, \frac{s-1}{m}\right\}. 
\end{equation}
Our choices of $m$ and $s$ here are arbitrary and we could have considered larger or smaller $m$ provided that $s$ is sufficiently small compared to $m$. 

\begin{figure}[h]
	\centering
	\begin{subfigure}{0.49\textwidth}
		\includegraphics[width=0.95\textwidth]{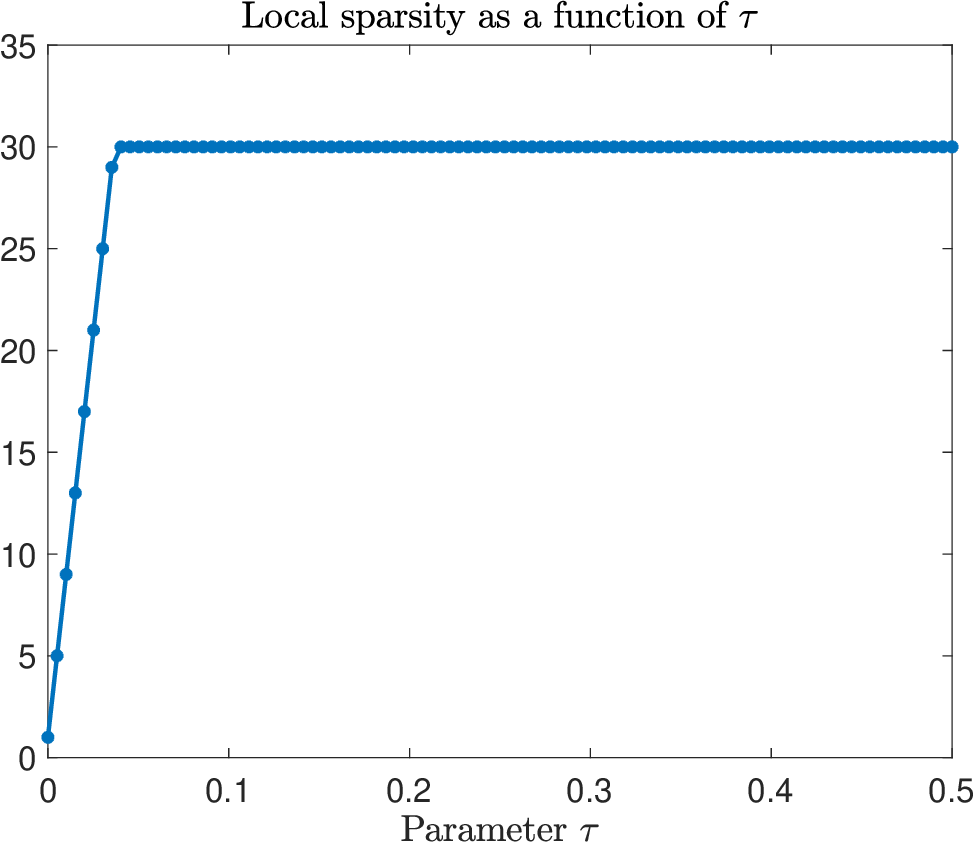}
	\end{subfigure}
	\begin{subfigure}{0.49\textwidth}
		\includegraphics[width=\textwidth]{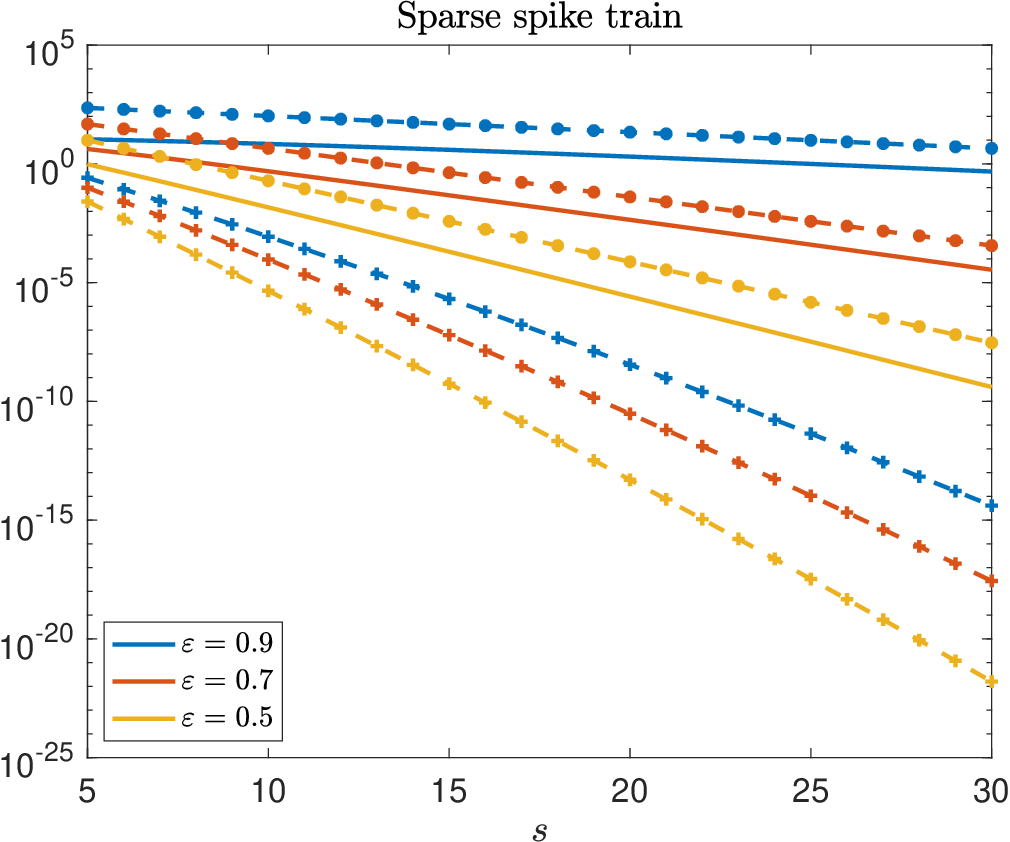}
	\end{subfigure}
	\caption{Left: Plot of the local sparsity $\nu(\tau,\calX_{s,\epsilon})$ as a function of $\tau$ for $\epsilon=\frac 12$. Right: For $\calX_{s,\epsilon}$ defined in \eqref{eq:spiketrain}, $m=200$, and $\epsilon=0.9,0.7,0.5$, the graphs of $\sigma_s(\Phi(m,\calX_{s,\epsilon}))$, upper bound \eqref{eq:barnettupper} with $C=500$, and lower bound \eqref{eq:main1} are shown in solid, dashed circular markers, and dashed plus sign markers, respectively. The slopes of the upper bounds are $-\frac {\pi (1-\epsilon)s}2$. The slopes of the lower bounds are $-1.2729$, $-1.5259$, and $-1.8625$, found by best linear fit.}
	\label{fig:spiketrain}
\end{figure}

It was shown in \cite{moitra2015matrixpencil} that for fixed $\epsilon$, provided that $s\geq C \log m$, then $\sigma_s(\Phi(m,\calX_{s,\epsilon}))\leq C' e^{-c\epsilon s}$ for some unspecified universal $C,C',c>0$. An improved estimate and with explicit $C,C',c$ were derived in \cite{barnett2022exponentially}. Provided that $\frac m \epsilon$ is an integer, \cite[Theorem 10]{barnett2022exponentially} provides an explicit upper bound, which for sufficiently large $\frac m\epsilon$, simplifies to
	\begin{equation}
		\label{eq:barnettupper}
		\sigma_s(\Phi(m,\calX_{s,\epsilon}))
		\leq C e^{- \pi (1-\epsilon) s/2}. 
	\end{equation}
	There is another bound \cite[Lemma 19]{barnett2022exponentially}, which is better for $\epsilon\leq \epsilon_*\approx 0.117$ and takes the form
	\begin{equation*}
		\sigma_s(\Phi(m,\calX_{s,\epsilon}))\leq \frac{2 \sqrt{ms}}{1-(e\pi (m-1)\epsilon /4m)} \bigg(\frac{e \pi(m-1)\epsilon}{4m}\bigg)^{s-1}.
	\end{equation*}
	These upper bounds have important implications. First, they show that even if $\calX$ consists of clumps, they need to be sufficiently far apart for the lower bound in \eqref{eq:clumps} to be valid, otherwise there is a contradiction. However, it does not provide a quantitative bound on the clump separation. Second, they imply that if $\Delta(\calX)<\frac 1 m$, then we need to put some restrictions on $s$ otherwise $\kappa(\Phi)$ may grow exponentially. Indeed, without an upper bound on $s$, we can let $\calX$ be the set in \eqref{eq:spiketrain} with $s$ on the order of $m$. This also explains why we cannot substantially relax the $(m,\tau)$ density criterion. We will provide more details related to the second point below. 

Notice that $\nu(\frac 1 m, \calX_{s,\epsilon})\leq 3$ for all $\epsilon\in (0,\frac 1 2]$, so at first glance, it may be temping to set $\tau$ on the order of $\frac 1 m$. However, it is not hard to see that there is no $\tau<\frac{3s}{m}$ for which $\calX_{s,\epsilon}$ satisfies the $(m,\tau)$ density criterion. On the other hand, if $\tau\geq \frac{3s}{m}$, then $\nu(\tau,\calX_{s,\epsilon})=s$ and consequently, $\calX$ satisfies the $(m,\tau)$ density criterion. The graph of $\nu(\tau,\calX_{s,\epsilon})$ as a function of $\tau$ is shown in Figure \ref{fig:spiketrain}. Thus, we are in the extreme case where we should just pick $\tau=\frac 1 2$. Intuitively, we think of $\calX_{s,\epsilon}$ as a high density set. 

Figure \ref{fig:spiketrain} plots the numerically computed $\sigma_s(\Phi(m,\calX_{s,\epsilon}))$ and our main theorem as functions of $s$ and for $\epsilon=0.9,0.7,0.5$. The numerical simulations indicate that $\log(\sigma_s (\Phi(m,\calX_{s,\epsilon})))$ is well approximated by an affine function of $s$, namely,
$$
\log(\sigma_s(\Phi(m,\calX_{s,\epsilon})))
= b_{\epsilon,m}-c_{\epsilon,m} s + \text{less significant terms depending on $s$}. 
$$
This is consistent with \cref{thm:main}, which predicts that the dominant term in $\log(\sigma_s(\Phi))$ is affine in $s$. To see why, using the notation defined in the theorem, there is a $k$ for which $\calI_k=\calX\setminus\{x_k\}$, so there is at least one term on the right side of \eqref{eq:main1} that contains a product of $s-1$ terms. They are the dominant terms since $\calI_k$ exerts the greatest influence on the lower bound.

If one wants a provable (but worse) result for this example, we recommend \cref{thm:main2}. Following the notation of that theorem, whenever $s\geq 2$, $m\geq 6s$, and $\epsilon \in (0,1]$, we set $\tau=\frac 12$ and $\delta=\frac \epsilon m$. Then for each $k$, we have $\calG_k=\emptyset$, $\nu(\tau,\calG_k)=0$, $n_k=m$, $\calB(x_k,\tau,\calX)=\calX$, and $r_k=s$. Since $m\geq 6s$, we have $\phi(\frac ms)\leq 1+\frac 1 5$. Using \eqref{eq:main3}, we have 
	\begin{equation}
		\label{eq:spiketrain2}
		\sigma_s(\Phi(m,\calX_{s,\epsilon}))
		\geq \frac{\pi}{2e} \sqrt{\frac{m}{\phi(\frac m s)}} \bigg( \frac{\sin(\frac{\pi \epsilon}2)}{2e \phi(\frac m s)}\bigg)^{s-1}
		\geq \frac{\pi}{2e} \sqrt{\frac{5m}{6}} \bigg( \frac{5\sin(\frac{\pi \epsilon}2)}{12e}\bigg)^{s-1}.
	\end{equation}
	The lower bound for $\log(\sigma_s(\Phi))$ provided by \eqref{eq:spiketrain2} has slope $-1.8879$, $-1.9909$ and $-2.2220$ for $\epsilon=0.9,0.7,0.5$ respectively.

\subsection{Colliding clumps}

Here we introduce an example where there are two localized sets that are progressive being pushed towards each other. To make this notion more precise, we fix $m = 100$ and for sufficiently small $\beta>0$, define 
\begin{equation}
	\label{eq:Xbeta}
	\calC_1 := \{0, \tfrac 1{2m}, \tfrac 2 {2m}\}, \quad 
	\calC_2(\beta):=\beta + \tfrac 1 m +\calC_1, \andspace \calX(\beta):=\calC_1\cup \calC_2(\beta).
\end{equation}
As $\beta\to 0$, the two sets $\calC_1$ and $\calC_2(\beta)$ become closer and $\sigma_s(\Phi(m,\calX(\beta)))\to 0$ as $\beta\to 0$. Note we can think of $\calC_1$ and $\calC_2(\beta)$ as clumps with separation $\beta$. Eventually for sufficiently small $\beta$, we should think of $\calX(\beta)$ as just a single clump as opposed to two separate ones. 

\begin{figure}[h]
	\centering 
	\includegraphics[width=0.49\textwidth]{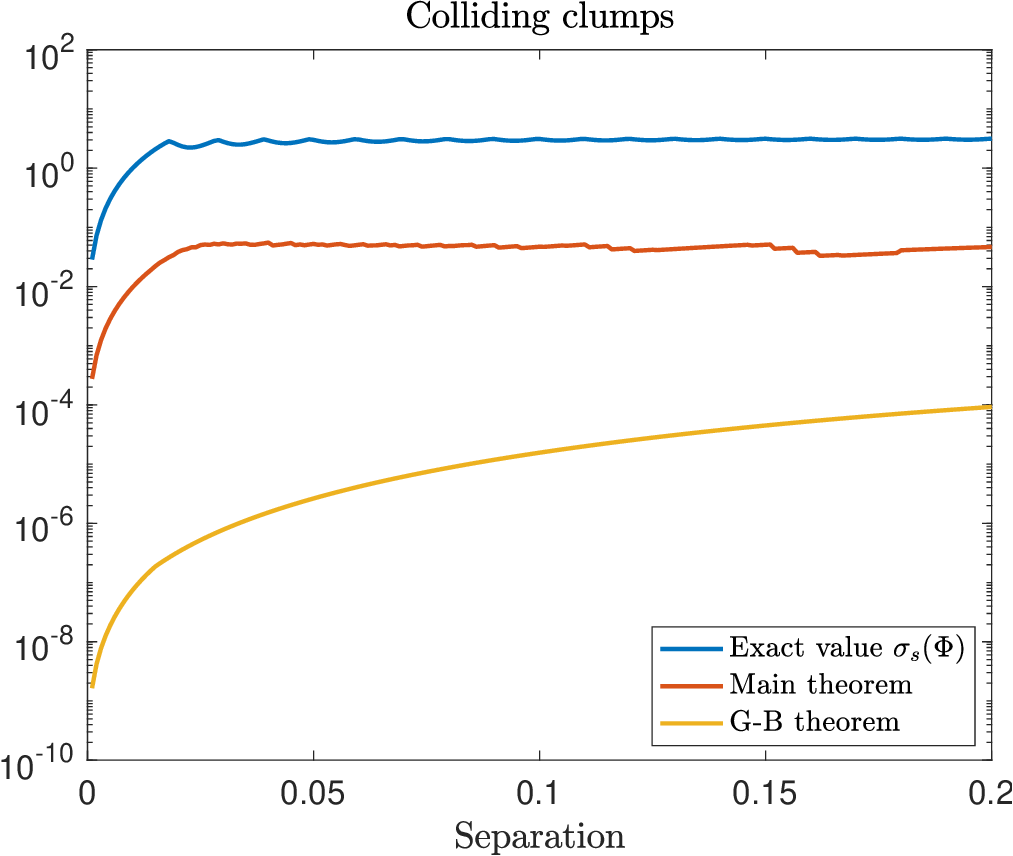}
	\caption{For $m=100$ and $\calX(\beta)$ defined in \eqref{eq:Xbeta}, plot of $\sigma_s(\Phi(m,\calX(\beta)))$, our main theorem, and the Gautschi-Baz\'an bound as a function of $\beta\in [\frac{0.1}{m},\frac{20}{m}]$. }
	\label{fig:collidingclumps}
\end{figure}

To employ \cref{thm:main}, we pick $\tau=\beta$ if $\beta\geq \frac{18}{m}$, otherwise we set $\tau=\frac 1 2$. Our choice of $\tau$ is consistent with \cref{prop:mainclumps}, but we picked a bigger constant (18 instead of 9) in front of $\frac 1 m$ to temper the growth of several implicit constants in our estimates. For this example, the clumps bounds \cite{li2021stable,kunis2020smallest} do not apply since $\calC_1$ and $\calC_2(\beta)$ are too close together. The results of this experiment are shown in Figure \ref{fig:collidingclumps}. This behavior of $\sigma_s(\Phi(m,\calX(\beta)))$ undergoes a phase transition at $\beta\approx \frac 1 m$ since for $\beta\gg \frac 1 m$, it is intuitive that $\calX(\beta)$ should be treated as two clumps instead of just one. There are some fluctuations in the graph of $\sigma_s(\Phi)$ due to number theoretic reasons since $\calX(\beta)$ is a partial unitary matrix if it is a subset of the lattice with spacing $\frac 1 {2m}$. Our estimate is significantly better than the Gautschi-Baz\'an theorem. For example, at $\beta = 0.1=\frac {10}{m}$, the former has an inaccuracy factor of 66.1225, while the latter is 1.9916e+05.

\section{Proof strategy}

\label{sec:proofstrat}

\subsection{The polynomial method}

The torus is defined as $\T:=\R/\Z$, which we normally identify with $[0,1)$ via the map $x\mapsto \text{mod}(x,1)$. The canonical basis vectors for $\R^d$ is denoted $\{e_k\}_{k=1}^d$. We let $\|\cdot\|_p$ and $\|\cdot\|_{L^p}$ denote the $\ell^p$ and $L^p$ norms respectively, for $1\leq p\leq \infty$. The Fourier transform of a $f\in L^2(\T)$ is denoted $\hat f\colon\Z\to\C$, where $\hat f(k):=\int_\T f(x) e^{-2\pi i kx}\, dx$ for each $k\in\Z$. We say $f$ is a trigonometric polynomial of degree $m-1$ if its Fourier transform is supported in $\{0,1,\dots,m-1\}$, and we let $\calP_m$ be the set of all trigonometric polynomials of degree at most $m-1$. 

The primary method that we will use to lower bound $\sigma_s(\Phi)$, or more precisely, upper bound $1/{\sigma_s(\Phi)}$, is through a ``dual" relationship with minimum norm trigonometric interpolation. This duality was introduced in \cite[Proposition 2.12]{li2021stable}: For any integers $m\geq s\geq 1$, finite set $\calX\subset \T$ of cardinality $s$, and unit vector $v\in \C^s$ such that $\|\Phi(m,\calX) v\| = \sigma_s(\Phi(m,\calX))$ (i.e., $v$ is any right singular vector corresponding to the smallest singular value of $\Phi(m,\calX)$), we have
\begin{equation}
	\label{eq:duality}
	\sigma_s(\Phi(m,\calX)) = \max\left\{ \frac{1}{\|f\|_{L^2}}\colon f\in \calP_m \text{ and } f(x_k)=v_k \text{ for each } k \right\}. 
\end{equation}
We refer to this equation as the {\it duality principle}, since it provides a connection between the smallest singular value to minimum norm trigonometric interpolation. There are related concepts \cite{donoho1992superresolution,beurling1989interpolation,beurling1989balayage} for $\R$ instead of $\T$, but one main difference is that our $\calX$ is arbitrary and can be completely nonuniform.  

It may be helpful to explain the intuition behind this duality principle. Suppose $v$ is a unit norm right singular vector corresponding to the smallest singular value of $\Phi(m,\calX)$ and we examine all solutions $u$ to $\Phi^* u= v$. Due to the singular value decomposition, any minimum norm vector $u$ that is consistent with this system will have norm $1/\sigma_s(\Phi)$. Note that $\Phi^*$ is the matrix representation of the linear transform that maps Fourier coefficients of functions in $\calP_m$ to their restriction on $\calX$. Using the Plancherel's theorem allows us to pass from the Fourier coefficients to polynomials. It follows from this discussion that a $f\in \calP_m$ which achieves equality in \eqref{eq:duality} is necessarily a $f$ whose Fourier coefficients are $u/\sigma_s(\Phi)$ where $u$ is any unit norm left singular vector of $\Phi$ which corresponds to $\sigma_s(\Phi)$.

The duality principle provides a natural and constructive avenue for lower bounding $\sigma_s(\Phi)$. However, since we have no exploitable information on the right singular vectors of $\Phi$, we construct interpolants for arbitrary $v$, and then estimate them in  $L^2$ uniformly in $v$. This leads us to the subsequent definition and lemma. 

\begin{definition}
	For any set $\calX=\{x_k\}_{k=1}^s\subset \T$, we say $\{f_k\}_{k=1}^s$ is a family of Lagrange interpolants for $\calX$ if $f_k(x_\ell)=\delta_{k,\ell}$ for each $1\leq k,\ell\leq s$.
\end{definition}

\begin{lemma}
	\label{lem:duality2}
	For any $m,s\in\N_+$ with $m\geq s$ and $\calX\subset\T$ of cardinality $s$, if $\{f_k\}_{k=1}^s\subset \calP_m$ is a family of Lagrange interpolants for $\calX$, then 
	$$
	\frac{1}{\sigma_s^2(\Phi(m,\calX))}
	\leq \sum_{k=1}^s \|f_k\|_{L^2}^2.
	$$
\end{lemma}

\begin{proof}
	Let $v\in\C^s$ be any unit norm vector such that $\|\Phi v\| = \sigma_s(\Phi)$. Since $f=\sum_{k=1}^s v_k f_k$ interpolates $v$ on $\calX$, by equation \eqref{eq:duality} and Cauchy-Schwarz, we have
	\begin{equation*}
		\frac{1}{\sigma_s(\Phi)}
		\leq \|f\|_{L^2}
		\leq \sum_{k=1}^s |v_k| \|f_k\|_{L^2}
		\leq \bigg(\sum_{k=1}^s \|f_k\|_{L^2}^2\bigg)^{1/2} \bigg(\sum_{k=1}^s |v_k|^2\bigg)^{1/2}
		=\bigg(\sum_{k=1}^s \|f_k\|_{L^2}^2\bigg)^{1/2} .
	\end{equation*}
\end{proof}

This lemma was implicitly used in \cite{li2021stable}, and allows us to reduce the problem of lower bounding $\sigma_s(\Phi)$ into constructing Lagrange interpolants. One strength of this method is that it does not require any information about the singular vectors of $\Phi$, which is usually more difficult to analyze than the singular values. However, if we had additional information about them, such as localization properties, then the $L^2$ estimate provided here can be improved. 

The next proposition serves as a converse to \cref{lem:duality2}. It shows that any lower bound on the smallest singular value provides the existence of polynomials with prescribed interpolation properties. 

\begin{proposition}
	\label{prop:interpolation2}
	If $\calX\subset\T$ is a non-empty finite set with cardinality $s$, then for any $w\in\C^s$ and integer $m\geq s$, there exits $f\in \calP_m$ such that $f|_\calX=w$, 
	$$\|f\|_{L^2}\leq \frac{\|w\|_2}{\sigma_s(\Phi(m,\calX))} \andspace \|f\|_{L^\infty}\leq \frac{ \sqrt{m} \, \|w\|_2}{\sigma_s(\Phi(m,\calX))}.$$ 
\end{proposition}

\begin{proof}
	Note that $m\geq s$ implies $\Phi:=\Phi(m,\calX)$ is injective due to the Vandermonde determinant theorem. We have the singular value decomposition $\Phi =\sum_{k=1}^s \sigma_k u_k v_k^*$, where the $v_k$'s and $u_k$'s are orthonormal and the $\sigma_k$'s are the nonzero singular values of $\Phi$. 
	
	For any $w\in\C^s$, we have $w=\sum_{k=1}^s b_k v_k$ for some $b\in \C^s$ such that $\|b\|_2=\|w\|_2$. For each $k$, we define $f_k\in \calP_m$ such that $\hat{f_k}=u_k/\sigma_k$. Using again that $\Phi^*$ is the matrix representation of the operator that maps the Fourier coefficients of a function in $\calP_m$ to its values on $\calX$, a direct calculation then yields that $f_k|_\calX=\Phi^* \hat{f_k}=v_k$.
	
	From here we see that $f:=\sum_{k=1}^s b_k f_k$ satisfies $f|_{\calX}=w$. Moreover, since the $u_k$'s are orthonormal, an application of Parseval's shows that the $f_k$'s are $L^2$ orthogonal, and so  
	$$
	\|f\|_{L^2}^2
	= \sum_{k=1}^s |b_k|^2 \|f_k\|_{L^2}^2
	=\sum_{k=1}^s  \frac{|b_k|^2}{\sigma_k^2}
	\leq \frac{\|w\|_2^2}{\sigma_s^2}.
	$$
	For the $L^\infty$ bound, we use that $f\in \calP_m$, and Cauchy-Schwarz, to get 
	$$
	\|f\|_{L^\infty}
	\leq \|\hat f\|_{\ell^1}
	\leq \sqrt m \, \|\hat f\|_{\ell^2}
	=\sqrt m \, \|f\|_{L^2}.
	$$
\end{proof}

This proposition can be rephrased as a result for an interpolation operator. Consider the operator $T(\calX)$ that takes $f^*\in C(\T)$ and produces the $f\in \calP_m$ guaranteed by this proposition such that $f^*=f$ on $\calX$. Equipping both $C(\T)$ and $\calP_m$ with the $L^\infty$ norm yields the estimate
$$
\|T(\calX)\|_{L^\infty\to L^\infty}
\leq \frac{ \sqrt{ms}}{\sigma_s(\Phi(m,\calX))}.
$$

We loosely refer to the strategy provided by the results of this subsection as the {\it polynomial method}. A primary usefulness of this connection between $\sigma_s(\Phi)$ and trigonometric interpolation is that it can be used employ tools from Fourier analysis and polynomial approximation, instead of solely working with matrices. While this connection is helpful, it is only useful if one can construct Lagrange interpolants with small norm, otherwise the resulting lower bounds for $\sigma_s(\Phi)$ would be quite loose. 

\subsection{Outline of the main proofs from an abstract perspective}

The proofs of Theorems \ref{thm:main} and \ref{thm:main2} are based on the following general recipe. Due to the polynomial method, we only need to provide the existence of Lagrange interpolants with suitably small norms and degree at most $m-1$. Note that $\calP_m$ enjoys numerous algebraic properties. In addition to being vector space, if $f\in \calP_m$ and $g\in \calP_n$, then $fg\in \calP_{m+n-1}$. It is also a shift invariant space, namely, $f\in \calP_m$ if and only if $f(\cdot-t)\in\calP_m$ for any $t\in\T$. 

We start the proof by fixing any $x_k\in\calX$ and concentrate on establishing a polynomial $f_k\in\calP_m$ such that $f_k(x_k)=1$ and vanishes on $\calX\setminus\{x_k\}$. Constructing a Lagrange interpolant of this data is straightforward, but doing so in a naive manner leads to loose estimates. We use the standard Lagrange interpolant $\ell_k\in\calP_s$ as a benchmark. Note that it has a pointwise upper bound,
\begin{equation}
	\label{eq:lagrangebound}
	\|\ell_k\|_{L^\infty}
	\leq \prod_{j\not=k} \frac{2}{|e^{2\pi i x_k}-e^{2\pi i x_j}|}
	= 2^{s-1} \prod_{j\not=k} \frac{1}{|e^{2\pi i x_k}-e^{2\pi i x_j}|}.
\end{equation}
The right hand side grows exponentially in $s$ and it contains a product of $s-1$ terms. It is significantly larger compared to the norms of polynomials that we will construct later. The main deficiency of $\ell_k$ is that $\deg(\ell_k)=s-1$, so it does not take advantage of the possibility that interpolants can be selected from $\calP_m$ where $m$ can be significantly larger than $s$. From this point of view, we interpret $m$ as the number of parameters or degrees of freedom, and $s$ as the number of constraints. 

Addition of two polynomials results in polynomial whose degree is the max, while multiplication adds their degrees. It is intuitive that points in $\calX$ near $x_k$ require larger norm polynomials to interpolate, since we need $f_k(x_k)=1$, yet $f_k$ can potentially have many nearby zeros, whereas further away points require smaller norms. Hence, it is natural to decompose
$$
\calX
=\calB_k\cup \calG_k,
$$
where the ``bad" $\calB_k$ and ``good" $\calG_k$ sets contain the points near and far away from $x_k$, respectively. The scale $\tau$ at which these sets are selected is important and determined by the density criterion. Hence the original interpolation problem can be solved by finding and multiplying Lagrange interpolants $b_k$ and $g_k$ where $b_k(x_k)=g_k(x_k)$, $b_k$ vanishes on $\calB_k\setminus\{x_k\}$, and $g_k$ vanishes on $\calG_k$. 

The interpolation problem for the good set will be handled in \cref{sec:good}. Although each element of $\calG_k$ is sufficiently far away from $x_k$, points in $\calG_k$ can still be close together. Hence, it is not clear that there is even any advantage of splitting $\calX$ into the good and bad sets. To deal with this, we will employ \cref{prop:decomp} to further decompose $\calG_k$ as
$$
\calG_k=\calG_{k,1}\cup \dots \cup \calG_{k,\nu_k}, \wherespace \nu_k:=\nu(\tau,\calG_k),
$$
such that $\Delta(\calG_{k,j})$ is suitably controlled from below. By using \cref{prop:interpolation2}, we can recast inequality \eqref{eq:wellseparated} as an interpolation statement. Doing so, we obtain the existence of $\nu_k$ many interpolants, which are multiplied together to obtain a desired $g_k$. Interpolation for the good set will require a budget of roughly ${2\nu(\tau,\calG_k)}/\tau$, which is guaranteed to be at most ${2m}/3$ in view of the density criterion.  

The interpolation problem for the bad set will be handled in \cref{sec:bad}. The starting point is a basic observation that the standard Lagrange interpolant for the bad set can be pointwise bounded by the distances between elements of $\calB_k$ and $x_k$ as seen in \eqref{eq:lagrangebound}. Note that if $q|t|_\T\leq \frac 1 2$ for some $q\in\N_+$, then $|qt|_\T=q|t|_\T$. Hence, if we shift and dilate the elements of $\calB_k$, and use a Lagrange interpolant for the dilated points, such as
$$
b_k(x)=\prod_{x_j\in\calB_k\setminus \{x_k\}} \frac{e^{2\pi iq_jx}-e^{2\pi i q_jx_j}}{e^{2\pi i q_jx_k}-e^{2\pi i q_jx_j}},
$$
then this polynomial will have significantly smaller norm and larger degree compared to the standard Lagrange interpolant $\ell_k$. Here, each $q_j$ will need to be chosen so that the degree of $b_k$ is not too large. Interpolation for the bad set will use the remaining portion of our budget consisting of roughly $m-1-{2\nu(\tau,\calG_k)}/\tau$. 

Finally, the desired Lagrange interpolant is $f_k:=b_kg_k$. Doing this for each $x_k\in\calX$ yields a family of Lagrange interpolants for $\calX$ in $\calP_m$, allowing us to employ \cref{lem:duality2}, which completes the proof. 

Carrying out these steps requires exploiting the advantages of several seemingly disparate approaches. The polynomial method for estimating the smallest singular value of Fourier matrices was introduced in \cite{li2021stable} and was inspired by interpolation techniques \cite{donoho1992superresolution,beurling1989interpolation}. It is further refined in this paper to handle more abstract sets, beyond clumps and subsets of lattices, by incorporating density ideas. Although we were unable to find a prior reference that uses exactly the same density criterion, there are strong connections between sampling and density, such as \cite{landau1967necessary}. 

The initial decomposition of $\calX$ into $\calB_k\cup \calG_k$ and further decompositions of $\calG_k$ into $\calG_{k,1},\dots,\calG_{k,\nu_k}$, are inspired by the classical Calder\'on-Zygmund decomposition. Our method for dealing with the good set requires the lower bound in \eqref{eq:wellseparated}, which was proved in \cite{aubel2019vandermonde} by using powerful machinery developed for analytic number theory \cite{vaaler1985some,selberg1989collected,montgomery1973large,montgomery1978analytic}. Finally, the method for dealing with the bad set using local dilation methods was originally employed in \cite{li2021stable}, for which we make significant improvements to.

\section{Two trigonometric interpolation problems} 
\label{sec:interpolation}

\subsection{Small norm Lagrange interpolants}

\label{sec:good}

In this section, we study the interpolation problem for the ``good" set $\calG\subset\T$ where all elements of $\calG$ are away from zero, and we would like to find a trigonometric polynomial that vanishes on $\calG$ and equals 1 at 0. Since we do not want to place any assumptions on $\Delta(\calG)$, which we allow to be arbitrarily small, this is a delicate problem. We will construct a polynomial that is significantly better behaved than the standard Lagrange interpolant. 

A key observation is the following {\it sparsity decomposition} which essentially states that a set can be decomposed into disjoint sets, each with unit local sparsity and minimum separation that is well-controlled. The key is that the number of sets equals the local sparsity of the original set, and not the cardinality. While this decomposition is intuitive, some care is taken with the proof due to the periodic boundary conditions that are imposed on us due to working with the torus. 

\begin{proposition}
	\label{prop:decomp}
	For any $\tau \in (0,\frac 1 2]$ and non-empty $\calW\subset\T$, letting $\nu:=\nu(\tau,\calW)$, there exist non-empty disjoint subsets $\calW_1,\calW_2,\dots,\calW_\nu\subset\calW$ such that their union is $\calW$ and $\Delta(\calW_k)>\tau$ for each $k$.
\end{proposition}

\begin{proof}
	Since the statement we are proving is invariant under periodic shifts, we can assume that $w_1:=0\in \calW$. We sort the elements of $\calW$ by $w_1,w_2,\dots,w_n$ sorted counterclockwise and provide a greedy method for generating the desired sets
	$
	\calW_1,\calW_2,\dots,\calW_\nu.
	$
	We initialize these sets to be empty and we add $w_\ell\in \calW$ to one of these sets until all elements of $\calW$ have been exhausted. We say $w_\ell$ has been {\it assigned} if it has been placed in a $\calW_k$, and {\it unassigned} otherwise. We start by placing $w_1\in \calW_1$. For each unassigned $w_\ell\in \calW$, we consider the set of $\calU_\ell:=\calW \cap [w_\ell-\tau,w_\ell+\tau]$ and place $w_\ell$ in an arbitrary $\calW_k$ that does not contain any assigned elements in $\calU_\ell$. This is always possible since $|\calU_\ell|\leq \nu$ for all $\ell$. By construction, $\calW_1,\dots,\calW_\nu$ are disjoint and  $\Delta(\calW_k)>\tau$. To see why $\calW_1,\calW_2,\dots,\calW_\nu$ are each nonempty, by definition of the $\tau$ density, there is a $\ell$ such that $[w_\ell-\tau,w_\ell+\tau]$ contains exactly $\nu$ elements of $\calW$ and they are necessarily placed in different $\calW_k$'s. 
\end{proof}

There is a stark conceptual distinction between the clumps decomposition in \cref{def:clumpsdecomp}, which groups the elements of $\calX$ by their spatial locations, versus \cref{prop:decomp}, which decomposes $\calX$ into disjoint subsets that each have unit local sparsity. An example is shown in Figure \ref{fig:clumpsvssparsity}. 

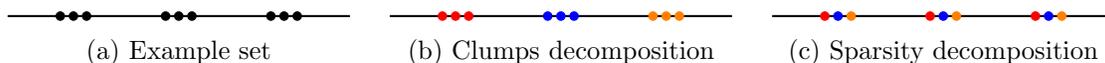
\begin{figure}[h]
	\centering
	\begin{subfigure}{0.45\textwidth}
		\centering
		\begin{tikzpicture}[scale = 1, 
			place2/.style={diamond,fill=blue,inner sep = 0.6mm},
			place4/.style={rectangle,fill=orange,inner sep = 0.9mm},
			place3/.style={regular polygon, regular polygon sides=3,fill=red,thick,inner sep = 0.5mm}]
			\draw[thick] (-6.5,0) -- (0,0);
			\node at (-5,0) [place3] {};
			\node at (-5.4,0) [place3] {};
			\node at (-5.8,0) [place3] {};
			\node at (-3.4,0) [place2] {};
			\node at (-3,0) [place2] {};
			\node at (-2.6,0) [place2] {};	
			\node at (-1,0) [place4] {};
			\node at (-0.6,0) [place4] {};
			\node at (-1.4,0) [place4] {};
			\node at (-6,2) [place3] {};
			\node[right,color=red] at (-5.9,2) {$\calC_1$};
			\node at (-6,1.5) [place4] {};
			\node[right,color=orange] at (-5.9,1.5) {$\calC_2$};
			\node at (-6,1) [place2] {};
			\node[right,color=blue] at (-5.9,1) {$\calC_3$};
			\draw[thick] (-6.4,.75) -- (-6.4,2.25) -- (-5.1,2.25) -- (-5.1,.75)--(-6.4,.75);
		\end{tikzpicture}
		\caption{Clumps decomposition}
	\end{subfigure}
	\begin{subfigure}{0.45\textwidth}
		\centering
		\begin{tikzpicture}[scale = 1, 
			place2/.style={diamond,fill=blue,inner sep = 0.6mm},
			place4/.style={rectangle,fill=orange,inner sep = 0.9mm},
			place3/.style={regular polygon, regular polygon sides=3,fill=red,thick,inner sep = 0.5mm}]
			\draw[thick] (-6.5,0) -- (0,0);
			\node at (-5,0) [place2] {};
			\node at (-5.4,0) [place3] {};
			\node at (-5.8,0) [place4] {};
			\node at (-2.6,0) [place2] {};
			\node at (-3,0) [place3] {};
			\node at (-3.4,0) [place4] {};	
			\node at (-0.6,0) [place2] {};
			\node at (-1,0) [place3] {};
			\node at (-1.4,0) [place4] {};
			\node at (-6,2) [place3] {};
			\node[right,color=red] at (-5.9,2) {$\calW_1$};
			\node at (-6,1.5) [place4] {};
			\node[right,color=orange] at (-5.9,1.5) {$\calW_2$};
			\node at (-6,1) [place2] {};
			\node[right,color=blue] at (-5.9,1) {$\calW_3$};
			\draw[thick] (-6.3,.75) -- (-6.3,2.25) -- (-5.1,2.25) -- (-5.1,.75)--(-6.3,.75);
		\end{tikzpicture}
		\caption{Sparsity decomposition}
	\end{subfigure}
	\caption{Clumps versus sparsity decomposition of the same set.}
	\label{fig:clumpsvssparsity}
\end{figure}

The usefulness of this decomposition for controlling the condition number of Fourier matrices is not obvious, but it will be made more clear in the following proof. 

\begin{proposition}
	\label{prop:goodset}
	Let $\calG\subset \T$ be a non-empty finite set such that for some $\tau \in (0,\frac 1 2]$, we have $|w|_\T > \tau$ for all $w\in \calG$. Suppose $m,r\in\N_+$ such that $\nu(\tau, \calG)\leq r$ and $m > \frac 1 \tau$. Then there is $f\in \calP_{r(m-1)+1}$ such that $f(0)=1$, $f$ vanishes on $\calG$, and 
	$$
	\|f\|_{L^2}\leq \frac{1}{\sqrt m} \bigg(1+\frac{1}{m\tau-1}\bigg)^{r/2}
	\andspace \|f\|_{L^\infty}\leq \bigg(1+\frac{1}{m\tau-1}\bigg)^{r/2}.
	$$ 
\end{proposition}	

\begin{proof}
	Let $\nu:=\nu(\tau,\calG)$. By \cref{prop:decomp}, there exists a disjoint decomposition,
	$$
	\calG=\calG_1\cup \calG_2\cup \cdots \cup \calG_\nu, \wherespace
	\calG_k\not=\emptyset \andspace \Delta(\calG_k)>\tau.
	$$
	The assumption that $|w|_\T > \tau$ for all $w\in\calG$ implies $\Delta( \calG_k\cup \{0\})> \tau$. Using the assumption $m>\frac 1\tau$, we invoke the lower bound in \eqref{eq:wellseparated}, which implies  
	$$
	\sigma_{\min}\big(\Phi\big(m, \calG_k\cup \{0\}\big)\big)
	\geq \sqrt {m -\frac 1 \tau}.
	$$	
	By \cref{prop:interpolation2}, applied to the data points $(\{0\}\cup \calG_k, e_1)$, there exists a $f_k\in\calP_{m}$ such that 
	\begin{equation}
		\label{eq:fk}
		f_k(0)=1, \quad f_k|_{\calG_k}=0, \quad
		\|f_k\|_{L^2}\leq  \frac{1}{\sqrt m} \sqrt{\frac{m\tau}{m\tau-1}} \andspace
		\|f_k\|_{L^\infty} \leq \sqrt{\frac{m\tau}{m\tau-1}}.
	\end{equation}
	Let $f$ be the product of $f_1,f_2,\dots,f_\nu$. It follows immediately from \eqref{eq:fk} that $f(0)=1$ and $f|_{\calG}=0$. The claimed bounds for $\|f\|_{L^2}$ and $\|f\|_{L^\infty}$ follow from H\"older's inequality. Moreover, we readily see that
	$$
	\deg(f)=\sum_{k=1}^\nu \deg(f_k)
	\leq \nu (m-1)
	\leq r(m-1).
	$$
\end{proof}

These polynomials can be numerically computed. First, we compute the decomposition of $\calG$ outlined in \cref{prop:decomp}, which can be done constructively using the greedy method described in its proof. Second, for each $\calG_k$ in this decomposition, we find an interpolant $f_k$ of the data $(\{0\}\cup \calG_k, e_1)$ via \cref{prop:interpolation2}. This can also be done numerically since $\hat{f_k}$ is precisely a scaled left singular vector of $\Phi(m,\{0\}\cup \calG_k)$, see the discussion following \eqref{eq:duality}. Finally, these interpolants are then multiplied together to yield the desired $f$.

We refer to a $f$ generated by this proposition as a {\it small norm Lagrange interpolant}. While each $f_k$ is found by minimizing a $L^2$ norm with interpolation constraints, it is not necessarily true that $f$ is also a minimum $L^2$ norm interpolant. Nonetheless, it is the pointwise bound that is important for this paper, and it is not clear if any of the $f_k$'s or $f$ are extremal in the $L^\infty$ norm. 

The interpolant in \cref{prop:goodset} enjoys many favorable and surprising properties. First, in the absence of additional assumptions, it is degree-optimal. Notice that $\Delta(\{0\}\cup \calG_k)>\tau$ and $0\not\in \calG$ imply that $(|\calG_k|+1)\tau <1$. This in turn establishes 
$$
|\calG|
\leq \nu \max_{1\leq k\leq \nu} |\calG_k|
\leq r \bigg\lfloor \frac{1}{\tau}-1 \bigg\rfloor.
$$
This inequality is sharp since it is possible to provide an example of a $\calG$ such that these inequalities are achieved. On the other hand, with the only stipulation that $m>\frac 1 \tau$, the theorem provides a polynomial of degree $r(m-1)$ that has up to $r\lfloor \frac{1}{\tau}-1\rfloor$ zeros. Hence it is not possible to reduce the degree of this interpolant in general. 

Aside from degree optimality, there is a second significance of small norm Lagrange interpolants, and to explain this, let us look at an example. For any $\epsilon>0$ that will be made sufficiently small, consider the set, 
\begin{equation}
	\label{eq:Gep}
	\calG_{\epsilon}
	= \bigg(\frac 1 4 +\{0,\epsilon,2\epsilon\}\bigg)\cup \bigg(\frac 1 2 +\{0,\epsilon\}\bigg) \cup \bigg( \frac 3 4 + \{0,\epsilon,\dots,3\epsilon\}\bigg).
\end{equation}
Whenever $\epsilon$ is small enough, this set consists of 9 elements. In order to find a Lagrange interpolant that interpolates $(0,1)$ and vanishes on $\calG_\epsilon$, we could use the usual Lagrange interpolant $\ell$ of degree $|\calG_\epsilon|=9$, which satisfies the inequality
$$
\|\ell \|_{L^\infty}
\leq 2^{|\calG_\epsilon|} \prod_{w\in\calG_\epsilon} \frac 1 {|1-e^{2\pi i w}|}.
$$
One advantage of the Lagrange interpolant is that its degree does not depend on $\epsilon$. However, $\|\ell\|_{L^\infty}$ grows exponentially  in $|\calG_\epsilon|$ and the product term may be large, which makes it rather unappealing for our purposes; this could be an artifact of estimating its sup-norm by upper bounding each term individually, but it is difficult to circumvent.

On the other hand, there are interpolation methods that yield interpolants with larger degrees but smaller norms. For instance, \cite[Theorem 2.1]{chandrasekaran2013minimum} proves that there is a trigonometric interpolant with good control over its norm, but requires the interpolant to have degree that scales inversely proportional to the minimum separation of the nodes. Hence, this result gives us an interpolant of any data defined on $\{0\}\cup \calG_\epsilon$ with degree that proportional to $\frac 1 \epsilon$. Related interpolation results in \cite{narcowich2004scattered} that are proved using functional analysis also exhibit similar behavior. 

In contrast to the above types of interpolants, the small norm Lagrange interpolants enjoy both advantages. Set $\tau=\frac 1 5$ and for all sufficiently small $\epsilon$, say $\epsilon\leq \frac 1 {300}$, we have $\nu(\tau,\calG_\epsilon)=4$. Pick $m=\frac 2 \tau = 10$ so that \cref{prop:goodset} is applicable. Notice that $\deg(f_\epsilon)=4\cdot 9 =36$ does not depend on $\epsilon$, hence does not explode as $\epsilon\to 0$. We also have $\|f_\epsilon\|_{L^\infty}\leq 2^{2}=4$ for all $\epsilon$. This is because the norm of $f$ does not grow exponentially in $|\calG_\epsilon|$, but only grows exponentially in $\nu(\tau,\calG_\epsilon)$. This is crucial for the purposes of this paper, since we do not want high degree interpolants or large norms, while allowing the minimum separation to be arbitrary. Graphs of the real part of the four interpolants $f_1,f_2,f_3,f_4$ generated by \cref{prop:goodset} are shown in Figure \ref{fig:smallnormL}.

\begin{figure}[h]
	\centering
	\includegraphics[height=0.25\textheight]{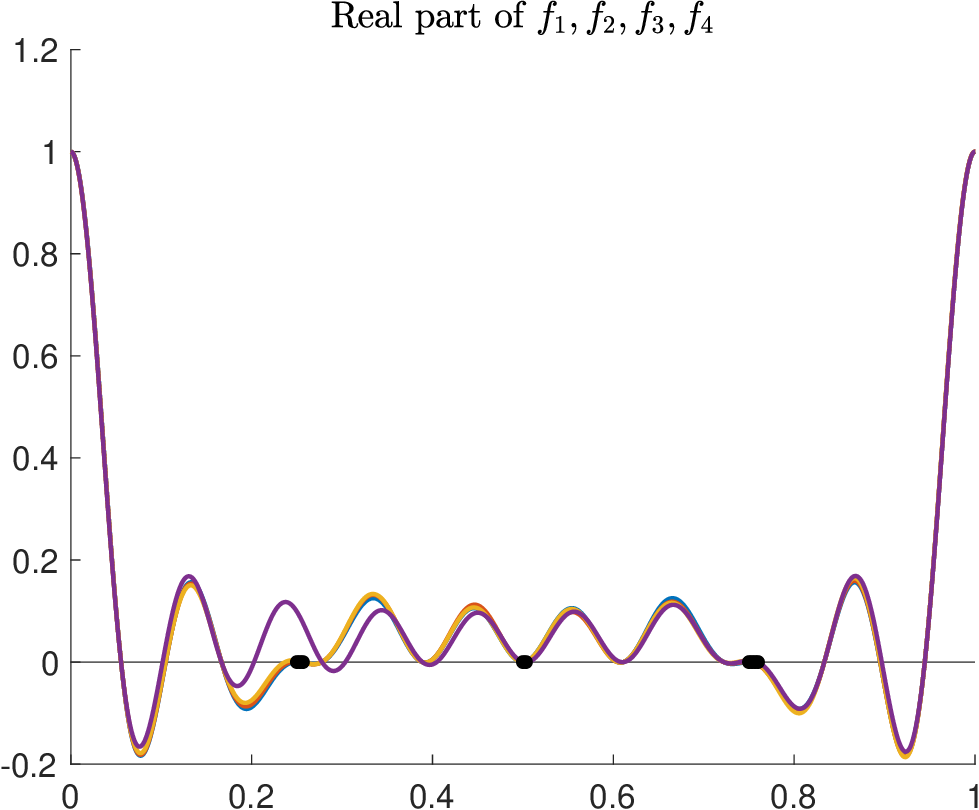} \ \ \ 
	\includegraphics[height=0.25\textheight]{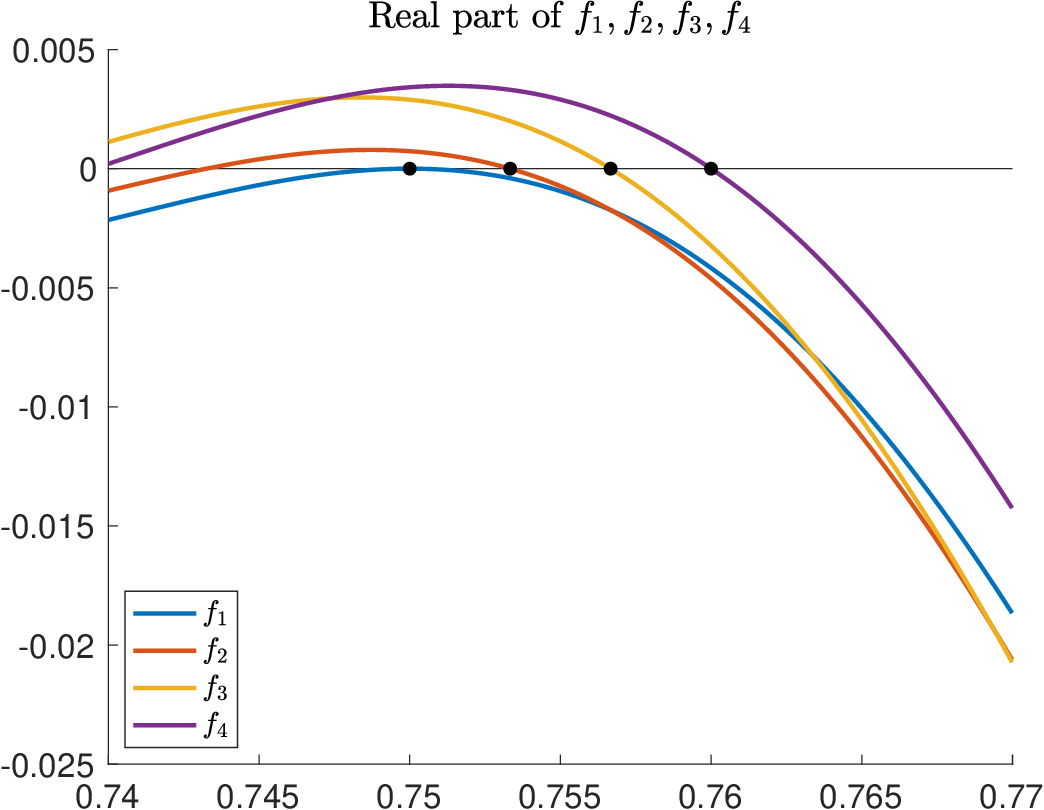}
	\caption{For the zero set $\calG_\epsilon$ defined in \eqref{eq:Gep} with $\epsilon=\frac 1 {300}$, the four small norm Lagrange interpolants each of degree $9$ are displayed.}
	\label{fig:smallnormL}
\end{figure}

We suspect that the advantages offered by \cref{prop:goodset} is highly dependent on the signs of the data that are being interpolated, that is, $1,0,\dots,0$ on $\{0\}\cup \calG_\epsilon$. The number `0' is not necessarily crucial though since we can interpolate $1,a,\dots,a$ as well by multiplying the small norm Lagrange interpolants by $1-a$ and then adding $a$. The important part here is that $0,\dots,0$ are all equal. In contrast, the other interpolation methods work for arbitrary data, which is perhaps why they do not enjoy the same advantages.

\subsection{Construction of local Lagrange polynomials}
\label{sec:bad}

In this section, we study the interpolation problem for the ``bad" set $\calB$. It contains $0$, all other elements in $\calB$ are close to zero, and we place no assumptions on $\Delta(\calB)$. We seek a trigonometric polynomial $f$ such that $f(0)=1$ and $f$ vanishes on $\calB\setminus \{0\}$. 

Throughout this paper, let $\psi\colon [-\frac 1 2,\frac 1 2]\to\R$ be the sinc kernel restricted to $[-\frac 1 2, \frac 1 2]$, 
	\begin{equation}
		\label{eq:psidef}
		\psi(t):=\begin{cases}
			\ \frac {\sin(\pi t)}{\pi t} &\text{if } t\not=0,\\
			\ \ \ \ 1 &\text{if } t=0.
		\end{cases}
	\end{equation}
	One can provide lower bounds for $\psi$ via Taylor expansions, but the point here is that multiplication by $\psi(t)$ is negligible whenever $t\approx 0$. It naturally appears from the following calculation. For all $|t|\leq \frac 1 2$, we have
\begin{equation}
	\label{eq:lawcos}
	|1-e^{2\pi it}|^2
	=2-2\cos(2\pi t)
	= 4 \pi^2 t^2 \bigg( \frac{ \sin (\pi t)}{\pi t} \bigg)^2
	= 4 \pi^2 t^2 \psi(t)^2.
\end{equation}
We have the basic bound that $\psi(t)\geq \psi(\frac 1 2) = \frac 2 \pi$ since it is decreasing away from zero in its domain. 

We have our first result for the bad set, which will be used in the proof of \cref{thm:main}.

\begin{lemma}
	\label{lem:badset1}
	Suppose $\calB$ is a set of at most $r$ points in $\T$ with $0\in \calB$ and $\calB\subset [-\tau,\tau]$ for some $\tau\in (0,\frac{1}{2}]$. For any $n\in \N_+$ such that $n\geq r$, define the subsets 
	$$
	\calI:=\left\{w\in \calB\colon 0<|w|_\T \leq \frac {r} {2n} \right\} \andspace
	\calJ:=\calB\setminus (\calI\cup \{0\}). 
	$$
	Then there exists a $f\in \calP_n$ such that $f$ vanishes on $\calB\setminus\{0\}$, $f(0)=1$, and 
	$$
	\|f\|_{L^2}
	\leq \frac{ 1}{\sqrt{\lfloor \frac n r \rfloor} } \prod_{w\in\calJ} \frac{1}{ 2\lfloor \frac 1 {2 |w|_\T} \rfloor |w|_\T} \prod_{w\in\calI}  \frac{1}{2 \lfloor \frac n r \rfloor |w|_\T}.
	$$
\end{lemma}

\begin{proof}
	We first deal with the $\calB=\{0\}$ case, in which case $\calI=\calJ=\emptyset$. Then we define $f$ to be a normalized Dirichlet kernel,
	$$
	f(x):= \frac{1}{n} \sum_{\ell=0}^{n-1} e^{2\pi i \ell x}.
	$$
	Notice that $f(0)=1$ and $f\in \calP_{n}$. Moreover, an application of Parseval establishes that for any $r\in\N_+$, we have 
	$$
	\|f\|_{L^2}\leq \frac 1 {\sqrt n}\leq \frac{ 1}{\sqrt{\lfloor \frac n r \rfloor} }.
	$$
	This takes care of the $\calB=\{0\}$ case. 
	
	From here onward, assume that $|\calB|\geq 2$. 	We first deal with the case that $\calI\not=\emptyset$. For each $w\in\calB$, we define the natural number $q(w)$ where 
	$$
	q(w):=q:= \left\lfloor \frac n r \right\rfloor \quad\text{if } w\in\calI, \andspace
	q(w):= \left\lfloor \frac 1 {2 |w|_\T} \right\rfloor\quad\text{if } w\in \calJ.
	$$
	We readily verify that $q(w)|w|_\T\leq \frac 1 2$ for all $w\in \calI\cup\calJ$, which implies 
	\begin{equation}
		\label{eq:qdilation}
		|q(w) w|_\T= q(w) |w|_\T \foreachspace w\in\calI\cup\calJ.
	\end{equation}
	In particular, this implies $|q(w) w|_\T\not=0$ whenever $w\in\calI\cup\calJ$. This enables us to define the polynomials,
	$$
	h_0(x)
	:=\prod_{w\in\calI} \frac{e^{2\pi i q(w) x}-e^{2\pi i q(w)w}}{1-e^{2\pi i q(w)w}}, \andspace
	g(x):=\prod_{w\in\calJ} \frac{e^{2\pi i q(w) x}-e^{2\pi i q(w)w}}{1-e^{2\pi i q(w)w}}.
	$$
	By construction, $h_0$ vanishes on $\calI$ and $h_0(0)=1$. For each $\ell=1,2,\dots,q-1$, we define the functions
	$$
	h_\ell(x):=e^{2\pi i\ell x}h_0(x)\andspace  h:= \frac 1 q \sum_{\ell=0}^{q-1} h_\ell. 
	$$
	Using that $|h_\ell|=|h_0|$, we see that $h(0)=1$ and $h$ also vanishes on $\calI$. Note that $h_\ell$ is a trigonometric polynomial whose frequencies are in $\ell+\{0,q,\dots |\calI|q\}$. This implies that $\{h_\ell\}_{\ell=0}^{q-1}$ are $L^2$ orthogonal, $\deg(h)\leq |\calI|q+q-1$, and by orthogonality,
	\begin{equation}
		\label{eq:hbound} 
		\| h \|_{L^2}^2
		=\frac{1}{q^2} \sum_{\ell=1}^{q-1} \|h_\ell\|_{L^2}^2
		= \frac{1}{q^2} \sum_{\ell=1}^{q-1} \|h_0\|_{L^2}^2
		\leq \frac 1 q \|h_0\|_{L^\infty}^2. 
	\end{equation}
	
	We define $f:=gh$. By construction, $f(0)=1$ and $f$ vanishes on $\calB\setminus\{0\}$. We bound the degree of $f$. Note for each $w\in \calJ$, we have $|w|_\T>\frac r n$ and so $q(w)\leq \frac{n}{2r}$. This implies, together with the assumption $|\calB|\leq r$, that
	\begin{align*}
		\deg(f)
		=|\calI|q + q-1 + \sum_{w\in \calJ} q(w)
		\leq \frac{n(|\calI|+1) }{r} -1 + \frac{n|\calJ| }{r}
		\leq \frac{n|\calB| }{r} -1
		\leq n-1.
	\end{align*}

	It remains to obtain the desired bound for $\|f\|_{L^2}$. Using \eqref{eq:hbound}, we get 
	\begin{align*}
		\|f\|_{L^2}
		\leq \|h\|_{L^2}\|g\|_{L^\infty}
		\leq \frac{1}{\sqrt q}\|h_0\|_{L^\infty} \|g\|_{L^\infty}
		\leq \frac{1}{\sqrt q}\prod_{w\in\calI\cup\calJ} \frac{2}{|1-e^{2\pi i q(w)w}|}.
	\end{align*} 
	Combining this with \eqref{eq:lawcos}, \eqref{eq:qdilation}, and that $\psi(t)\geq \frac 2 \pi$ for $|t|\leq \frac 1 2$, we obtain 
	\begin{align*}
		\|f\|_{L^2}
		\leq \frac{1}{\sqrt q} \prod_{w\in\calB} \frac{1}{\pi |q(w) w|_\T \, \psi(|q(w)w|_\T)}
		\leq \frac{1}{\sqrt q} \prod_{w\in\calB} \frac{1}{2 q(w) |w|_\T}.
	\end{align*}
	Using the definition of $q(w)$ yields the claimed upper bound for $\|f\|_{L^2}$ when $\calI\not=\emptyset$. 
	
	Finally, for the remaining case where $|\calB|\geq 2$ and $\calI=\emptyset$, we use the same $g$ as above but with a different $h$. We instead use a normalized Dirichlet kernel, $$h(x):=\frac{1}{\lfloor n/r \rfloor} \sum_{\ell=0}^{\lfloor n/r \rfloor-1} e^{2\pi i \ell x}.$$
	We have $h\in \calP_{\lfloor n/r \rfloor}$ and that $\|h\|_{L^2}= 1/ \sqrt{\lfloor n/r \rfloor}$. Setting $f:=gh$, we see that 
	$$
	\deg(f)
	=\deg(g)+\deg(h)
	\leq |\calJ| \frac{n}{r}+ \frac n r -1
	\leq n-1, 
	$$
	where the final inequality used that since $\calI=\emptyset$ and $0\in\calB$, we have $|\calJ|=|\calB|-1\leq r-1$. Hence, $f\in \calP_n$ and it satisfies the desired interpolation properties. Combining the previous upper bounds for $\|g\|_{L^\infty}$ and $\|h\|_{L^2}$ completes the proof.
		
\end{proof}

The following is our second result for the bad set, which will be used in the proof of \cref{thm:main2}. 

\begin{lemma}
	\label{lem:badset2}
	Let $n,r\in \N_+$ such that $n\geq r$ and $\delta\in (0,\frac 1 n]$. Suppose $\calB\subset\T$ is a finite set such that $0\in \calB$, $|\calB|=r$ and $\delta\leq \Delta(\calB)\leq \frac 1 n$. Then there exists a $f\in \calP_n$ such that $f$ vanishes on $\calB\setminus \{0\}$, $f(0)=1$, and 
	$$
	\|f\|_{L^2}
	\leq \frac{2e}{\pi r}  \frac {1}{\sqrt{\lfloor \frac n r \rfloor}} \left( \frac{4e}{ \pi \psi(\frac {n \delta}{2}) r \lfloor \frac n r \rfloor \delta} \right)^{r-1}.
	$$
\end{lemma}

\begin{proof}
	We define the subsets,
	$$
	\calI:=\bigg\{w\in \calB\colon 0<|w|_\T \leq \frac {r} {2n} \bigg\} \andspace
	\calJ:=\calB\setminus (\calI\cup \{0\}). 
	$$ 
	We first deal with the $\calB=\{0\}$ case, in which case $r=1$ and $\calI=\calJ=\emptyset$. Similar to the proof of \cref{lem:badset1}, we define 
	$
	f(x):= \frac{1}{n} \sum_{\ell=0}^{n-1} e^{2\pi i \ell x}.
	$
	Then $f(0)=1$, $f\in \calP_{n}$, and 
	$
	\|f\|_{L^2}\leq \frac 1 {\sqrt n}.
	$
	Since $\frac{2e}{\pi}\geq 1$, this proves the $\calB=\{0\}$ case. 
	
	From here onward assume that $r\geq 2$. We enumerate the elements of $\calB$ as $0=w_0,w_1,\dots,w_{r-1}$ where $|w_k|_\T\leq |w_{k+1}|_\T$ for each $k$. 
	For reasons that will become apparent later, for any $d\in\N_+$, we define the following sequence, $0,-1,1,-2,2,-3,3,\dots$, which we enumerate by $a_0,a_1,\dots$. We define the natural numbers $q_1,\dots,q_{r-1}$ as 
	$$
	q_k:=q:=\left\lfloor \frac n r \right\rfloor \quad \text{if } w_k\in \calI, \andspace
	q_k :=  \left\lfloor \frac{q |a_k| \delta}{|w_k|_\T} \right\rfloor \quad \text{if } w_k\in \calJ. 
	$$
	We need to set the stage before we explicitly construct $f$. For each $w_k\in\calI$, we immediately get $q_k|w_k|_\T\leq \frac 1 2$ by definition of $q$ and $\calI$. For each $w_k\in \calJ$, we use that $|a_k|\leq \frac r 2$ regardless of the parity of $r$ and the assumption $\delta\leq \frac 1 n$ to see that $q_k|w_k|_\T \leq q |a_k|\delta	\leq \frac 1 2$. This implies 
	\begin{equation}
		\label{eq:dilationhelp}
		|q_kw_k|_\T = q_k|w_k|_\T \foreachspace w_k\in \calB\setminus \{0\}.
	\end{equation}
	This enables us to define the polynomials,
	$$
	h_0(x)
	:=\prod_{w\in\calI} \frac{e^{2\pi i q x}-e^{2\pi i q w}}{1-e^{2\pi i q w}}, \andspace
	g(x):=\prod_{w_k\in\calJ} \frac{e^{2\pi i q_k x}-e^{2\pi i q_k w_k}}{1-e^{2\pi i q_k w_k}}.
	$$ 
	We repeat the same sub-argument that appeared in the proof of \cref{lem:badset1}. We define the function $h(x):=\frac{1}{q}\sum_{\ell=0}^{q-1} e^{2\pi i \ell x}h_0(x)$, and we see that 
	$
	\|h\|_{L^2}
	\leq \frac 1 {\sqrt {q}} \|h_0\|_{L^\infty}. 
	$ Thus, we define $f=gh$ and so  
	\begin{equation}
		\label{eq:help3}
		\|f\|_{L^2}
		\leq \frac 1 {\sqrt {q}} \|h_0\|_{L^\infty} \|g\|_{L^\infty}
		\leq \frac 1 {\sqrt {q}} \prod_{w\in\calI} \frac{2}{|1-e^{2\pi i qw}|} \prod_{w_k\in\calJ} \frac{2}{|1-e^{2\pi i q_k w_k}|}.
	\end{equation}
	
	By construction, $f$ satisfies the desired interpolation properties. We argue that $f\in \calP_n$. Notice that $\deg(h_0)=q|\calI|$ and $\deg(h)=q|\calI|+q-1$. On the other hand, for each $w_k\in \calJ$, using that $|a_k|\leq \frac r 2$, $\delta\leq \frac 1 n$, and $|w_k|_\T >\frac r{2n}$, we see that $q_k \leq q$. Thus,  
	\begin{align*}
		\deg(f)
		= q(|\calI|+1)-1 + \sum_{w_k\in \calJ} q_k
		&\leq q (|\calI|+|\calJ|+1) - 1
		\leq \frac {n|\calB|} r - 1
		\leq n-1.  
	\end{align*}
	It remains to upper bound $\|f\|_{L^2}$. By \eqref{eq:dilationhelp}, we see that $|q_k w_k|_\T=q_k|w_k|_\T\geq q |a_k| \delta/2$. Using this inequality on the right side of \eqref{eq:help3}, we get 
	\begin{equation}
		\label{eq:hbound1}
		\|f\|_{L^2}
		\leq \frac 1 {\sqrt {q}} \prod_{w\in\calI} \frac{2}{\sqrt{2-2\cos (2\pi q|w|_\T)}} \prod_{w_k\in\calJ} \frac{2}{\sqrt{2-2\cos (2\pi q |a_k| \delta/2)}}.
	\end{equation}

	We let $\ell:=|\calI|$ and claim that
	\begin{equation}
		\label{eq:hbound3}
		\prod_{w\in\calI} \frac{2}{\sqrt{2-2\cos (2\pi q|w|_\T)}}
		\leq \prod_{k=1}^\ell \frac{2}{\sqrt{2-2\cos (2\pi q|a_k|\delta)}} 
	\end{equation}
	Recall that $\Delta(\calI\cup \{0\})\geq \delta$ and that $0\not\in\calI$. We define the auxiliary function, 
	$$
	\gamma(t_1,t_2,\dots,t_\ell):=\sum_{k=1}^\ell \frac{1}{1-\cos(2\pi q|t_k|)},
	$$
	where $\delta\leq |t_k|\leq \frac 1 {2q}$ for each $k$ and $|t_j-t_k|\geq \delta$ for each $j\not=k$. Clearly this function increases if any $t_j$ is made smaller while the remaining $t_k$'s are fixed. We claim that $\gamma$ is maximized precisely when $t_1,t_2,\dots,t_\ell$ is $a_1\delta,a_2\delta,\dots,a_\ell\delta$. To see this, we list $t_1,t_2,\dots,t_\ell$ as $\{u_1,\dots,u_a,v_1,\dots,v_b\}$ where $a+b=\ell$ and
	$$
	u_a < \cdots < u_1\leq -\delta <0<\delta\leq v_1<\cdots < v_b. 
	$$
	If $b\geq 1$, we can assume that $v_1=\delta$ since a shift of all $v_1,\dots,v_b$ by the same amount towards 0 increases the value of $\gamma$. If $a\geq 1$, we can likewise assume that $u_1=-\delta$. Finally, $\gamma$ is further increased if all the $u$'s and $v$'s are fixed except $v_2$ is replaced with $2\delta$, then $v_3$ is moved to $3\delta$, etc. Likewise, $\gamma$ is increased if $u_2$ is replaced to $-2\delta$, etc. Hence, we see that $\gamma(t_1,t_2,\dots,t_\ell)$ is dominated by $\gamma(-a\delta, \dots, -\delta, \delta,\dots,b\delta)$. If $|a-b|>1$, then by reflecting elements across the origin and shifting again, we see that $\gamma(-a\delta, \dots, -\delta, \delta,\dots,b\delta)$ is further dominated by $\gamma(a_1\delta,a_2\delta,\dots,a_\ell\delta)$. This establishes inequality \eqref{eq:hbound3}.
	
	We continue with the upper bound for $\|f\|_{L^2}$. Note that $q|a_k|\delta \leq \frac {n\delta} 2$ for each $k=1,\dots,r-1$ and that $\psi$ is decreasing on $[0,\frac 12 ]$. Using \eqref{eq:lawcos}, \eqref{eq:hbound1}, and \eqref{eq:hbound3}, we see that
	\begin{equation}
		\label{eq:hbound2}
		\|f\|_{L^2}
		\leq \frac 1 {\sqrt {q}} \prod_{k=1}^\ell \frac{1}{ \pi \psi(q|a_k|\delta) q|a_k|\delta } \prod_{w_k\in\calJ} \frac{2}{\pi \psi(q|a_k|\delta/2) q|a_k|\delta }
		\leq \frac 1 {\sqrt {q}}  \prod_{k=1}^{r-1} \frac{2}{\pi \psi(\frac {n \delta}{2}) q |a_k| \delta }.
	\end{equation}
	
	To control the product over $k$, first note that 
	$$
	C_d
	:=\prod_{k=1}^{d}|a_k|
	=
	\begin{cases}
		\ \ \ \ \big(\frac{d}{2} \, ! \, \big) \big(\frac{d}{2} \, ! \, \big) &\quad \text{if $d$ is even}, \medskip \\
		\ \big( \frac{d+1}{2} \, !\big) \big(\frac{d-1}{2} \, ! \, \big)  &\quad \text{if $d$ is odd}. 
	\end{cases}
	$$
	Recall the well known inequalities $k! \geq \sqrt{2\pi k} \, (\frac{k}{e})^k$ and $1+t\leq e^t$ for all $t\geq 0$. We have $C_1=1$, and for $d\geq 2$, we have
	\begin{align*}
		\frac {(d+1)^{d}} {C_d} 
		&\leq \frac{(2e)^{d}}{\pi d} \, \(\frac{d+1}{d}\)^{d}
		\leq \frac{(2e)^{d} e}{\pi d}
		\leq \frac{(2e)^{d} 2e}{\pi (d+1)} \quad \text{if $d$ is even}, \\
		\frac {(d+1)^{d}} {C_d} 
		&= \frac{2(d+1)^{d-1}} {(\frac {d-1}{2} !)^2} 
		\leq \frac{2(2e)^{d-1}}{\pi (d-1)} \(\frac{d+1}{d-1}\)^{d-1}
		\leq \frac{2(2e)^{d-1}e^2}{\pi (d-1)}
		\leq \frac{ (2e)^{d} 2e }{\pi (d+1)} \quad \text{if $d$ is odd}.
	\end{align*}
	Using these and the definition of $q$ in \eqref{eq:hbound2} completes the proof.  
\end{proof}

It may be worthwhile to mention a subtle technical part of this proof. We do not prove, or claim, that the ``worst case" $\calB$ (up to trivial invariances) satisfying the hypotheses of \cref{lem:badset2}  is 
\begin{equation}
	\label{eq:Bdelta}
	\delta \left\{-\left\lfloor \frac {r-1}2\right\rfloor,\dots,-1,0,1,\dots, \left\lceil \frac {r-1}2\right\rceil\right\}.
\end{equation}
In this proof, we showed that for any $\calB$ satisfying the hypotheses, there is an explicit Lagrange interpolant $f$ with some upper bound for $\|f\|_{L^2}$. Then we showed that this upper bound for $\|f\|_{L^2}$ is maximized precisely when $\calB$ is equal to \eqref{eq:Bdelta}. This is the quantity reported in the conclusion of \cref{lem:badset2}.

\section{Proofs of the main results}

\label{sec:proofs}

\subsection{Proof of \cref{thm:main}}

\begin{proof}
	
	Let us first set the stage and discuss several immediate implications of the assumptions. Fix any $1\leq k\leq s$ and for convenience, we define
	$$
	n_k:= \left\lfloor m- \frac{2\nu(\tau,\calG_k)} \tau \right\rfloor. 
	$$
	Note that $\nu(\tau,\calG_k)\leq \nu(\tau,\calX)$ since $\calG_k\subset\calX$. Also using the assumptions $m\geq 6s$ and ${3\nu(\tau, \calX)}  \leq \tau m$, we have
	\begin{equation}
		\label{eq:start}
		m- \frac {2 \nu(\tau,\calG_k)} \tau 
		\geq  m- \frac {2 \nu(\tau,\calX)} \tau 
		\geq \frac {m} 3
		\geq 2s.
	\end{equation}
	As immediate consequences of this inequality, we have $n_k\in\N_+$ and that $\alpha_k>0$. Since $\nu(\tau,\calB_k)\leq \nu(\tau,\calX)$ due to $\calB_k\subset\calX$, we use the assumption ${3\nu(\tau, \calX)}  \leq \tau m$ to see that 
	\begin{equation}
		\label{eq:alphatau}
		\alpha_k 
		\leq \frac{\nu(\tau,\calX)}{2m - 4 \nu(\tau,\calX) \tau^{-1}}
		\leq \frac{\nu(\tau,\calX)}{6\nu(\tau,\calX) \tau^{-1} - 4 \nu(\tau,\calX) \tau^{-1}}
		\leq \frac \tau 2.
	\end{equation}
	
	We first deal with the ``good" set $\calG_k$. If $\calG_k=\emptyset$, then $\nu(\tau,\calG_k)=0$ and we set $g_k:=1$. Otherwise, we assume $\calG_k\not=\emptyset$. We apply \cref{prop:goodset}, where $\calG_k-x_k$, $\lceil \frac 2 \tau \rceil$, and $\nu(\tau,\calG_k)$ play the roles of $\calB$, $m$ and $r$ respectively, in the referenced proposition's notation. This provides us with a polynomial, which after shifting by $x_k$, we call it $g_k\in \calP_{\nu(\tau,\calG_k)(\lceil \frac 2 \tau \rceil-1)+1}$ such that $g_k(x_k)=1$, $g_k$ vanishes on $\calG_k$, and 
	\begin{equation}
		\label{eq:gnorm}
		\|g_k\|_{L^\infty}
		\leq \(1+\frac{1}{\lceil \frac 2 \tau \rceil\tau-1}\)^{\nu(\tau,\calG_k)/2}
		\leq 2^{\nu(\tau,\calG_k)/2}.
	\end{equation}
	Note that this statement is still valid in the corner case that $\calG_k=\emptyset$ since in this case, we have $\nu(\tau,\calG_k)=0$, which is consistent with $\|g_k\|_{L^\infty}=1$ and $g_k\in \calP_1$.
	
	Now we deal with the ``bad" set $\calB_k:=\calB(x_k,\tau,\calX)$. We use the shorthand notation $r_k:=|\calB_k|$. We are ready to employ \cref{lem:badset1}, where $\calB_k-x_k$, $n_k$, and $r_k$ play the roles of $\calB$, $n$, and $r$ respectively, in the referenced lemma's notation. Note that $n_k\geq s$ from \eqref{eq:start}. The lemma provides us with a polynomial, and after shifting by $x_k$, we call it $b_k\in \calP_{n_k}$ such that $b_k$ vanishes on $\calB_k\setminus \{x_k\}$, $b_k(x_k)=1$, and $b_k$ enjoys the estimates,
	\begin{equation*}
		\|b_k\|_{L^2}
		\leq \sqrt{\frac{1} {\lfloor \frac{n_k}{r_k}\rfloor}} \prod_{x\in\calJ_k} \frac{1}{2\lfloor \frac 1 {2 |x-x_k|_\T} \rfloor |x-x_k|_\T} \ \prod_{x\in\calI_k} \frac{1}{ 2 \lfloor \frac{n_k}{r_k}\rfloor |x-x_k|_\T}.
	\end{equation*}	
	Now we perform some algebraic manipulations and simplifications. First note that $\frac 1 {\alpha_k}\leq \frac{2n_k+2}{r_k}$, and $\lfloor \frac {n_k}{r_k} \rfloor \geq \frac {n_k}{r_k}-\frac{r_k-1}{r_k}$ since $r_k\in\N_+$. Together, they imply that 
	$$
	\frac 1 {2 \lfloor \frac {n_k}{r_k} \rfloor}
	\leq \frac{1}{\frac 1 {\alpha_k}-2}
	=\frac{\alpha_k}{1-2\alpha_k}.
	$$
	Using this observation and the definition of $\phi$ in the previous upper bound for $\|b_k\|_{L^2}$, we have
	\begin{equation}
		\label{eq:bknorm}
		\|b_k\|_{L^2}
		\leq \sqrt{\frac{2\alpha_k}{1-2\alpha_k}} \prod_{x\in\calJ_k} \phi\( \frac{1}{2|x-x_k|_\T}\) \ \prod_{x\in\calI_k} \frac{\alpha_k}{(1-2\alpha_k)|x-x_k|_\T}.
	\end{equation}
	
	We next define $f_k:=b_k g_k$. We have $f_k\in\calP_m$ because  
	\begin{align*}
		\deg(f_k)
		&=\deg(b_k)+\deg(g_k) 
		\leq \nu(\tau,\calG_k) \(\left\lceil \frac 2 \tau \right\rceil-1\)+n_k-1
		\leq \frac {2 \nu(\tau,\calG_k)} \tau +n_k-1
		\leq m-1.
	\end{align*}
	Together with the interpolation properties of $g_k$ and $b_k$, we see that $\{f_k\}_{k=1}^s\subset \calP_m$ is a family of Lagrange interpolants for $\calX$. We use H\"older's inequality and the upper bounds \eqref{eq:gnorm} and \eqref{eq:bknorm} to get
	\begin{equation*}
		\label{eq:fnorm}    
		\|f_k\|_{L^2}
		\leq \sqrt{\frac{2\alpha_k 2^{\nu(\tau,\calG_k)}}{1-2\alpha_k}}  \prod_{x\in\calJ_k} \phi\( \frac{1}{2|x-x_k|_\T}\) \prod_{x\in\calI_k} \frac{\alpha_k }{(1-2\alpha_k)|x-x_k|_\T}. 
	\end{equation*}	
	We apply \cref{lem:duality2} to complete the proof of \eqref{eq:main1}.	

	Now we proceed to further upper bound the right side of \eqref{eq:main1}. Using that $\phi(t)\leq 2$ for all $t\geq 1$, that $\alpha_k\leq \frac \tau 2\leq \frac 1 4$ due to \eqref{eq:alphatau}, we obtain
	\begin{align*}
		\frac{1}{\sigma_s^2(\Phi(m,\calX))} 
		&\leq \sum_{k=1}^s \, 2^{\nu(\tau,\calG_k)} \, 4\alpha_k \,  4^{|\calJ_k|} \prod_{x\in\calI_k} \frac{4\alpha_k^2}{|x-x_k|_\T^2} \\
		&\leq s \ \max_{1\leq k\leq s} \ \Bigg\{ 2^{\nu(\tau,\calG_k)} 4\alpha_k 4^{|\calJ_k|} \prod_{x\in\calI_k} \frac{4\alpha_k^2}{|x-x_k|_\T^2} \Bigg\}.
	\end{align*}
	Rearranging this inequality and taking the square root completes the proof  of \eqref{eq:main2}. 
\end{proof}

\subsection{Proof of \cref{thm:main2}}

\begin{proof}
	Fix any $1\leq k\leq s$. Note \eqref{eq:start} showed that $n_k\geq \lfloor \frac m 3 \rfloor$, while $n_k\leq m$ immediately by definition. The proof is analogous to the proof of \cref{thm:main}, but with a different function for the bad set. We carry over the same definitions of $n_k$ and $r_k$. There we constructed the function $g_k\in \calP_{\nu(\tau,\calG_k)(\lceil \frac 2 \tau \rceil-1)+1}$ for the ``good" set $\calG_k$ such that 
	\begin{equation}
		\label{eq:gnorm2}
		\|g_k\|_{L^\infty}
		\leq 2^{\nu(\tau,\calG_k)/2}.
	\end{equation}
	
	For the ``bad" set $\calB_k:=\calB(x_k,\tau,\calX)$, we note  that $\Delta(\calB_k)\geq \Delta(\calX)\geq \delta$. We use \cref{lem:badset2}, where $\calB_k-x_k$, $n_k$ and $r_k$ play the roles of $\calB$, $n$, and $r$ respectively, in the referenced lemma's notation. Also note that $n_k\geq s\geq r_k$ due to \eqref{eq:start}, and that $\delta\leq \frac 1 m \leq \frac 1 {n_k}$. The lemma provides us with a polynomial, and after shifting by $x_k$, we call it $b_k\in \calP_{n_k}$ such that $b_k$ vanishes on $\calB_k\setminus \{x_k\}$, $b_k(x_k)=1$, and $b_k$ enjoys the estimate,
	\begin{equation}
		\label{eq:bnorm2}
		\|b_k\|_{L^2}
		\leq \frac{2e}{\pi r_k}  \frac {1}{\sqrt{\lfloor \frac {n_k} {r_k} \rfloor}} \( \frac{4e}{\pi \psi(\frac {n_k \delta}{2}) r_k \lfloor \frac {n_k} {r_k} \rfloor  \delta} \)^{r_k-1}. 
	\end{equation}
	
	We define $f_k=g_kb_k$. Repeating the same argument as in the proof of \cref{thm:main} shows that $f_k\in\calP_m$. By construction, $\{f_k\}_{k=1}^s$ is a family of Lagrange polynomials for $\calX$.  Using H\"older's inequality, \eqref{eq:gnorm2}, \eqref{eq:bnorm2}, and the definition of $\phi$, we get
	$$
	\|f_k\|_{L^2}
	\leq \frac{2e}{\pi} \sqrt{\frac {2^{\nu(\tau,\calG_k)} \phi(\frac {n_k} {r_k})}{r_kn_k}} \( \frac{4e \phi(\frac {n_k} {r_k})}{\pi \psi(\frac {n_k \delta}{2})n_k\delta} \)^{r_k-1}. 
	$$
	Finally, using \cref{lem:duality2} completes the proof of \eqref{eq:main3}. 

	We proceed to make numerous simplifications of the right hand side of \eqref{eq:main3}. Since $\delta\leq \frac 1 m$ and $n_k\leq m$, we have $\psi(\frac {n_k \delta}{2})\geq \psi(\frac 1 2)=\frac 2 \pi$. Note that \eqref{eq:start} and the assumption $m\geq 6s$ imply
	\begin{equation*}
		r_k \left\lfloor \frac{n_k}{r_k} \right\rfloor
		\geq r_k \( \frac{n_k}{r_k} - \frac{r_k-1}{r_k} \)
		=n_k-r_k+1
		\geq m- \frac {2 \nu(\tau,\calX)} \tau -r_k
		\geq \frac{m}{3} - s
		\geq \frac m 6. 
	\end{equation*}
	Using these observations in \eqref{eq:main3} now establishes 
	\begin{equation*}
		\frac{1}{\sigma_s^2(\Phi(m,\calX))}
		\leq \frac{24 e^2}{\pi^2m} \sum_{k=1}^s  \frac{2^{\nu(\tau,\calG_k)}}{r_k} \( \frac{12 e}{m \delta} \)^{2r_k-2}
		\leq \frac{24 e^2 s}{\pi^2m} \max_{1\leq k\leq s} \Bigg\{ \frac{2^{\nu(\tau,\calG_k)}}{r_k} \( \frac{12 e}{m \delta} \)^{2r_k-2} \Bigg\}. 
	\end{equation*}
	Rearranging this inequality and taking the square root completes the proof of \eqref{eq:main4}. 
\end{proof}

\subsection{Proof of \cref{prop:mainclumps}}

\begin{proof}
	We first claim that $\calX$ satisfies the $(m,\tau)$ density criterion. This trivially holds when $r=1$ because then $\tau = \frac 1 2$ and so $\nu(\tau,\calX)=|\calX|=s=\lambda$. From here onward, assume that $r>1$. For any $x\in\calX$, let $\calC_x$ be the clump that $x$ belongs to. Since $\beta>\alpha$ by definition, we see that 
	$$
	\calC_x = \{x'\in \calX\colon |x-x'|_\T\leq \beta\},
	$$
	otherwise it would contradict the assumption that any two clumps are separated by distances strictly larger than $\beta$ and that $\diam(\calC_x)\leq \alpha<\beta$. This shows that $\nu(\tau,\calX)\leq \lambda$, and since there is a clump that has cardinality exactly equal to $\lambda$, we see that $\nu(\tau,\calX)=\lambda$. We have $\beta\geq \frac{3\lambda}{m}$ by assumption, and so  
	$$
	\frac{3\nu(\tau,\calX)}{\tau}
	=\frac{3 \lambda}{\beta}
	\leq m.
	$$
	We have shown that $\calX$ satisfies the $(m,\tau)$ density criterion. This shows that the assumptions of \cref{thm:main} and \cref{thm:main2} are satisfied. For each $k$ in the right side of \eqref{eq:main4}, we use that $r_k\leq \nu(\tau,\calX)=\lambda$ and $\nu(\tau,\calG_k)\leq \nu(\tau,\calX)=\lambda$ to complete the proof. 
\end{proof}

\subsection{Proof of \cref{thm:upper}}

\begin{proof}
	Letting $\nu:=\nu(\tau,\calX)$, by the decomposition given in \cref{prop:decomp}, we have a disjoint union 
	$$
	\calX = \calX_1 \cup \calX_2 \cup \cdots \cup \calX_{\nu}, \wherespace \Delta(\calX_k)\geq \tau \foreachspace k=1,2,\dots,\nu.
	$$
	Since the singular values of $\Phi(m,\calX)$ are invariant under permutations of its columns, after reshuffling,
	$$
	\Phi(m,\calX)=\begin{bmatrix}
		\Phi(m,\calX_1) & \Phi(m,\calX_2) &\cdots &\Phi(m,\calX_\nu)
	\end{bmatrix}.
	$$
	Let $u\in \C^{|\calX|}$ be any unit norm vector, and likewise, we partition $u$ into sub-vectors $u_1,u_2,\dots,u_\nu$ such that $u_k\in \C^{|\calX_k|}$. Since $\Delta(\calX_k)\geq \tau$ for each $k$ and $m>\frac 1 \tau$, we use the upper bound in \eqref{eq:wellseparated} to get
	\begin{align*}
		\|\Phi(m,\calX) u\|_2
		=\left\| \sum_{k=1}^\nu \Phi(m,\calX_k) u_k \right\|_2
		\leq \sqrt {m + \frac 1 \tau} \ \sum_{k=1}^\nu \|u_k\|_2.
	\end{align*}
	Using Cauchy-Schwarz and that $u$ has unit norm, we obtain 
	$$
	\sum_{k=1}^\nu \|u_k\|_2
	\leq \sqrt \nu \, \( \sum_{k=1}^\nu \|u_k\|_2^2\)^{1/2}
	= \sqrt \nu \|u\|_2
	= \sqrt \nu.  
	$$
	Combining the above inequalities completes the proof. 
\end{proof}

\section*{Computational costs and benefits}

Selecting reasonable $\tau$ to use in the main theorems can be done computationally. Let $\calT\subset (0,\frac 12]$ be a collection of $\tau$ for which we would like to evaluate the main estimates. We proceed to analyze the naive time and storage complexity of evaluating the main estimates for each $\tau\in \calT$. All of the main inequalities are simple expressions that can be evaluated once the various parameters are determined, which importantly, only depend on $\tau$ and distances between elements in $\calX$. 

As overhead, we first compute the distances between all ${s\choose 2}$ distinct pairs of elements in $\calX$, which requires $O(s^2)$ operations and storage. As we will see, computation of the local sparsity is the main bottleneck. For an arbitrary and finite set $\calU\subset\T$, computation of $\nu(\tau,\calU)$ requires $O(|\calU|^2)$ operations since we need to enumerate through each $u\in \calU$ and find the number of $u'\in U$ such that $|u-u'|\leq \tau$.

Coming back to the task at hand, we perform a first loop over each $x_k\in \calX$ and $\tau\in \calT$, in order to compute $|\calB(x_k,\tau,\calX)|$ and $\nu(\tau,\calG(x_k,\tau,\calX))$. This loop requires $O(s^3|\calT|)$ time and $O(s|\calT|)$ storage. After the first loop executes, we perform a second loop through each $x_k\in \calX$ and $\tau\in \calT$. In this loop, we calculate the necessary parameters (e.g., $\alpha_k$, $n_k$) and each of the $s$ terms in the main inequalities. This second loop requires $O(s^2|\calT|)$ time for \cref{thm:main} and $O(s|\calT|)$ time for \cref{thm:main2}. Then we perform a final minimization or summation (depending on which inequality is evaluated), and then find the the optimal $\tau\in \calT$. Ultimately, computation of the optimal $\tau\in \calT$ for any of these theorems requires $O(s^3|\calT|)$ operations and $O(s^2+|\calT|)$ storage. 

Let us quickly remark that these complexities are for arbitrary $\calX$ and $\calT$, and the computations can be sped up under special cases. Suppose $\calX$ consists of clumps and $\calT \subset [\alpha,\beta]$, where $\alpha$ and $\beta$ are clump parameters as in \cref{def:clumpsdecomp}. Notice that $|\calB(x_k,\tau,\calX)|$ is the cardinality of the clump that contains $x_k$, while $\calG(x_k,\tau,\calX)$ consists of all the other clumps that do not contain $x_k$, and $\nu(\tau,\calG(x_k,\tau,\calX))$ is the cardinality of the largest clump that does not contain $x_k$. All of them do not vary with $\tau\in \calT$. So these quantities can be computed with $O(s^2)$ operations instead of $O(s^3|\calT|)$. 

Another possibility of improvement is to take advantage of redundancies in these calculations when looping through $\tau\in \calT$. Notice that $|\calB(x_k,\tau,\calX)|$ and $\nu(\tau,\calG(x_k,\tau,\calX))$ are piece-wise constant functions with at most $s-1$ break points, so there is no need to enumerate through $\calT$. So the $|\calT|$ terms in the above complexities can be replaced with $s^2$ whenever $|\calT|\geq s^2$. However, in our experience, we usually have a good sense for what range of $\tau$ is appropriate, so only a few different $\tau$ need to be evaluated. Moreover, the loop over $\calT$ can be parallelized to further reduce the scaling in $|\calT|$.

In comparison, the singular value decomposition of $\Phi(m,\calX)$ requires $O(m^2s)$ operations and $O(ms)$ storage; recall that $m\geq s$ and we require $m\geq 6s$ to use the main theorems. Depending on the relationship between $m$, $s$, and $|\calT|$, it may be significantly faster to compute the expressions of the main theorems to obtain approximations rather than numerically computing the actual smallest singular value. This gain is particularly noticeable for many signal and image processing problems where $m\gg s$. 

Our theorems offer significant computational benefits when there is uncertainty in $\calX$ or freedom to choose $\calX$. For instance, consider a situation where each $x_k$ lies in some interval $[x_k^*-\epsilon_k,x_k^*+\epsilon_k]$. This may occur if we only have some rough idea of what $x_k$ is or if we are forced to select $x_k$ in this interval due to some constraint. Hence, we are dealing with infinitely many possible sets $\calX$. In principle, we could $N$ many possible $\calX$, and use the SVD to calculate the singular values of $\Phi(m,\calX)$, which requires $O(Nm^2 s)$ operations.

On the other hand, our theorems give interpretable bounds for the smallest singular value in terms of the geometry of $\calX$, so we have some understanding of which sets lead to poor condition numbers. This greatly reduces the number of sets $\calX$ that would be considered, from large $N$ to much smaller $n$. Hence, it suffices to compute $n$ many SVDs, which requires $O(nm^2s)$ operations, or evaluate our main theorems for $n$ many $\calX$ and $|\calT|$ many $\tau$, which requires $O(ns^3|\calT|)$ operations. 

\section*{Conclusion and future work}

This paper presented multiscale estimates for the condition number of Fourier matrices for general $\calX$ provided that there is a modicum of redundancy, $m\geq 6s$. The main results are completely new whenever $\Delta(\calX)<\frac 1 m$ and $\calX$ does not consist of separated clumps. Even in the clump framework, the main results significantly reduce sufficient conditions of prior works and achieve similar estimates. The main results also greatly improve upon classical estimates and provide a unified framework for dealing with a disparate collection of sets, which were previously treated on a case-by-case basis.  

We state one immediate consequence of the main results. It was shown in \cite{li2020super} that the stability of a foundational algorithm called ESPRIT \cite{roy1989esprit} used for signal processing enjoys (under suitable conditions) the error estimate
$$
\text{error} \leq \frac{C_s \cdot \text{noise}}{\sigma_s^2(\Phi(m,\calX))}. 
$$ 
A significance of this inequality is that it establishes ESPRIT is near min-max optimal. Since this paper greatly enlarges the collection of $\calX$ for which we have accurate estimates for $\sigma_s(\Phi(m,\calX))$, it yields significant practical implications for ESPRIT and related signal processing algorithms and applications such as \cite{li2022stability}. These improvements and their implications will be discussed in a separate article. 

Returning to discussion of main results, a natural question is the selection of an optimal scale parameter $\tau$ for which to invoke the main inequalities in \cref{thm:main,thm:main2}. This is not a simple task (for theoretical purposes) and greatly depends on $\calX$. We saw examples where the best effective scale $\tau$ is on the order of $\frac 1 m$ such as for clumps, whereas $\tau=\frac 1 2$ for sparse spike trains. These polarizing examples illustrate that the optimal effective scale does not only depend on $m$, $s$, and/or $\Delta(\calX)$, but on more complicated relationships depending on $\calX$.

Regarding the main theorems' assumptions, they can be weakened to $m\geq 3s$ and
$
\frac{2\nu(\tau,\calX)}{\tau}+s \leq m-1
$
without significant modifications to the main proofs. However, doing so would change the numerical constants in a rather undesirable way. For this reason, we decided to state the main results with a stronger than necessary conditions. The techniques introduced in this paper are unable to deal with the extreme case where $m\geq s$ and $\frac{\nu(\tau,\calX)}{2\tau} \leq m$. This is due to splitting the good and bad sets into separate problems, which comes at a cost of making the interpolants' degrees larger than necessary. To circumvent this, one can handle the good and bad sets in a unified manner and construct interpolants in a completely different way. Our construction of these alternative polynomials have horribly large norms, which in turn, yields a lower bound for $\sigma_s(\Phi)$ that appears to have limited use outside of special contexts. 

Many of the techniques and ideas in this paper, including the polynomial method, are flexible. They can be altered to deal with more restricted classes of $\calX$ if desired and can be extended to multivariate Fourier matrices. Such a matrix has the form $\Phi=[e^{2\pi i j\cdot x_k}]_{j\in\Omega,\, x_k\in \calX}$ for some $\calX=\{x_k\}_{k=1}^s\subset \T^d$ and $\Omega \subset\R^d$. There are many open questions about the condition number of multivariate Fourier matrices and their behavior greatly depends on the structure of both $\calX$ and $\Omega$. From the dual perspective, interpolation by multivariate polynomials is also much more involved. Due to these added technical difficulties and important differences, we deal with multidimensional matrices in a separate article \cite{li2024nonharmonic}.

\section*{Acknowledgments}

WL is supported by NSF-DMS Award \#2309602, a PSC-CUNY grant, and a start-up fund from the Foundation for City College. The author thanks John J. Benedetto, Albert Fannjiang, Wenjing Liao, and Kui Ren for helpful feedback and suggestions. The author also thanks an anonymous reviewer for his/her careful reading and insightful suggestions that led to many improvements of the initial draft.

\bibliography{ImprovedFourierBib.bib}
\bibliographystyle{plain}

\end{document}